\documentclass[12pt,reqno,a4paper]{amsart}
\usepackage{amsmath,amsthm,amssymb,amsbsy}
\usepackage{a4wide}
\usepackage{hyperref}
\usepackage{enumerate}

\usepackage{color}
\usepackage{ifpdf}
\ifpdf %if using pdfLaTeX in PDF mode
  \usepackage[pdftex]{graphicx}
  \DeclareGraphicsExtensions{.pdf,.png,.mps}
  \usepackage{pgf}
  \usepackage{tikz}
\else %if using LaTeX or pdfLaTeX in DVI mode
  \usepackage{graphicx}
  \DeclareGraphicsExtensions{.eps,.bmp}
  \usepackage{pgf}
  \usepackage{tikz}
  \usepackage{pstricks}%variant: \usepackage{pst-all}
\fi
\usepackage{epic}
\usepackage{wrapfig}% package for wrapfigure environment

\makeatletter
 \def\markboth#1#2{%
  \begingroup
   \@temptokena{{#1}{#2}}\xdef\@themark{\the\@temptokena}%
   \mark{\the\@temptokena}%
  \endgroup
  \if@nobreak\ifvmode\nobreak\fi\fi}
 \def\thanks#1{\g@addto@macro\thankses{\thanks{#1}}
 }
 \makeatother
\newtheorem{thm}{Theorem}%[section]

\newtheorem{conj}[thm]{Conjecture}
\newtheorem{lem}[thm]{Lemma}

\theoremstyle{definition}

\newtheorem*{rmk}{Remark}
\newtheorem*{ex}{Example}
\newcommand{\nt}{\par\noindent{\it Note. }}
\newcommand{\qu}{\par\noindent{\it Question. }}
\newcommand{\subtitle}[1]{\par\medskip\noindent{\bf #1.}}
\newcommand{\norm}[1]{\left\Vert#1\right\Vert}
\newcommand{\abs}[1]{\left\vert#1\right\vert}
\newcommand{\set}[1]{\left\{#1\right\}}
\newcommand{\bl}{\vspace{-\baselineskip}}
\newcommand{\ds}{\displaystyle}
\newcommand{\ts}{\textstyle}

\def\eps{\varepsilon}
\def\To{\longrightarrow}
\def\Real{\mathbb R}
\def\f{\mathcal F}
\def\g{\mathcal G}
\def\h{\mathcal H}
\def\t{\mathcal T}
\def\Tn{\mathcal {T}_n}
\def\Tl{\mathcal {T}_{\lambda}}
\def\Tnr{\mathcal {T}_n^{(r)}}
\def\Snr{\mathcal {S}_n^{(r)}}
\def\Pinr{\Pi_n^{(r)}}
\def\Tnkr{\mathcal {T}_{n,k}^{(r)}}
\def\Snkr{\mathcal {S}_{n,k}^{(r)}}
\def\Pinkr{\Pi_{n,k}^{(r)}}
\def\Tlr{\mathcal {T}_{\lambda}^{(r)}}
\def\Tnl{\mathcal {T}_{n,\lambda}}
\def\Tnlr{\mathcal {T}_{n,\lambda}^{(r)}}
\def\Slr{\mathcal {S}_{\lambda}^{(r)}}
\def\Pilr{\Pi_{\lambda}^{(r)}}
\def\Tloc{\mathcal {T}_{loc}^{(\lambda)}}
\def\Tglo{\mathcal {T}_{glo}^{(r,\lambda)}}
\def\Tlocpi{\mathcal {T}_{loc}^{(\pi)}}
\def\Tglopi{\mathcal {T}_{glo}^{(\pi)}}
\def\philoc{\phi_{loc}}
\def\phiglo{\phi_{glo}}
\def\ord{\mathop {\rm ord}}
\def\deg{\mathop {\rm deg}}
\def\indeg{{\mathop {\rm indeg}}}
\def\root{{\mathop {\rm root}}}
\def\sib{{\mathop {\rm sibship}}}
\def\sibglo{{\mathop {\rm sibship}}_{glo}}
\def\sibloc{{\mathop {\rm sibship}}_{loc}}
\def\type{{\mathop {\rm type}}}
\def\typeglo{{\mathop {\rm type}}_{glo}}
\def\typeloc{{\mathop {\rm type}}_{loc}}
\newcommand\qbin[2]{{#1 \brack #2}_{q}}
\newcommand\qmultinomial[2]{{#1 \brack #2}_{q}}
\def\area{{\mathop {\rm area}}}

\begin{document}
\title{A bijective enumeration of labeled trees with given indegree sequence}

\author{Heesung Shin}
\address{Universit\'e de Lyon; Universit\'e Lyon 1; Institut Camille Jordan, CNRS UMR 5208; 43 boulevard du 11 novembre 1918, F-69622 Villeurbanne Cedex, France}
\email{hshin@math.univ-lyon1.fr}
%\thanks{This work was partially supported by the Korea Research Foundation Grant funded by the Korean Government(MOEHRD). KRF-2007-357-C00001}

\author{Jiang Zeng}
\address{Universit\'e de Lyon; Universit\'e Lyon 1; Institut Camille Jordan, CNRS UMR 5208; 43 boulevard du 11 novembre 1918, F-69622 Villeurbanne Cedex, France}
\email{zeng@math.univ-lyon1.fr}

\subjclass[2000]{05A15}%
\date{\today}
\maketitle

\begin{abstract}
For a labeled  tree on the vertex set $\set{1,2,\ldots,n}$, the local direction of each
edge $(i\,j)$ is  from $i$ to $j$ if $i<j$.  For a rooted tree, there is also a natural global direction of edges towards the root.
The number of edges pointing to a vertex is called its indegree. Thus
the local (resp. global) indegree sequence $\lambda = 1^{e_1}2^{e_2} \ldots$ of a tree
on the vertex set $\set{1,2,\ldots,n}$ is a partition of $n-1$.
We construct a bijection from (unrooted) trees to rooted trees
such that the local indegree sequence of a (unrooted) tree equals the global indegree sequence of the corresponding rooted tree.
Combining with a Pr\"ufer-like code for rooted labeled trees, we obtain a bijective proof of a recent conjecture by Cotterill and also
solve two open problems proposed by Du and Yin.
We also prove a $q$-multisum binomial coefficient identity  which confirms another conjecture of Cotterill in a very special case.
\end{abstract}

\tableofcontents
%%%%%%%%%%%%%%%%%%%%%%%%%%%%%%%%%%%%%%
\section{Introduction}
For an oriented tree $T$,  the {\em indegree} of a vertex $v$ is  the number of edges pointing to it and
 the sequence $(e_0, e_1, e_2, \ldots)$ is called the {\em type} of $T$
where  $e_h$ is the number of vertices of $T$ with  indegree $i$.
Since  $\sum_{i\ge 0} e_h$ (resp. $\sum_{i\ge 0} i e_h$)  is  the number of  vertices (resp. edges) of $T$,
we have $e_0 = 1+\sum_{i\ge 1} (i-1) e_h$. Hence
we can ignore $e_0$ while dealing  with types of trees because $e_{0}$ is determinated  by the others.
The partition $\lambda = 1^{e_1} 2^{e_2} \ldots$ will be called the {\em indegree sequence} of $T$.
Throughout this paper,  for any partition $\lambda = 1^{m_1} 2^{m_2} \ldots$,  we denote its length and weight by
$\ell(\lambda) = \sum_{i\ge 1} m_i$ and  $\abs{\lambda} = \sum_{i\ge 1} i m_i$.
Clearly,
if $\lambda$ is an indegree sequence of a tree  on $[n]:=\set{1,\ldots,n}$, then  $\abs{\lambda}=n-1$ and
$e_0 = \abs{\lambda} +1 - \ell(\lambda) = n - \ell(\lambda)$.

Let  $\Tn$  be the set of unrooted labeled  trees on $[n]$.
For any edge $(ij)$ of a  tree $T\in \Tn$,
there is a   {\em local orientation}, which
orients $(ij)$  towards its smaller vertex, i.e.,  $i \to j$ if $i<j$.
Let $\Tnr$ be the set of labeled trees on $[n]$ rooted at $r\in [n]$.
For any edge $(ij)$ of a  tree $T\in \Tnr$, there is a
{\em global orientation}, which  orients each edge  towards the root.
  It is interesting to note that for a rooted tree each edge has both a global orentation and
a local orientation.
An example of the local and global orientations is given in Figure~\ref{fig:indegree}.

\begin{figure}[t]
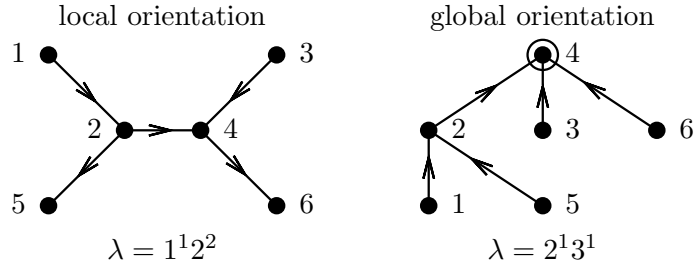
\label{fig:indegree}
\centering
\begin{pgfpicture}{13.00mm}{10.14mm}{107.00mm}{49.14mm}
\pgfsetxvec{\pgfpoint{1.00mm}{0mm}}
\pgfsetyvec{\pgfpoint{0mm}{1.00mm}}
\color[rgb]{0,0,0}\pgfsetlinewidth{0.30mm}\pgfsetdash{}{0mm}
\pgfcircle[fill]{\pgfxy(20.00,40.00)}{1.00mm}
\pgfcircle[stroke]{\pgfxy(20.00,40.00)}{1.00mm}
\pgfmoveto{\pgfxy(26.11,33.89)}\pgflineto{\pgfxy(30.00,30.00)}\pgfstroke
\pgfmoveto{\pgfxy(20.00,40.00)}\pgflineto{\pgfxy(26.11,33.89)}\pgfstroke
\pgfmoveto{\pgfxy(26.11,33.89)}\pgflineto{\pgfxy(24.63,36.36)}\pgflineto{\pgfxy(26.11,33.89)}\pgflineto{\pgfxy(23.64,35.37)}\pgflineto{\pgfxy(26.11,33.89)}\pgfclosepath\pgffill
\pgfmoveto{\pgfxy(26.11,33.89)}\pgflineto{\pgfxy(24.63,36.36)}\pgflineto{\pgfxy(26.11,33.89)}\pgflineto{\pgfxy(23.64,35.37)}\pgflineto{\pgfxy(26.11,33.89)}\pgfclosepath\pgfstroke
\pgfcircle[fill]{\pgfxy(30.00,30.00)}{1.00mm}
\pgfcircle[stroke]{\pgfxy(30.00,30.00)}{1.00mm}
\pgfcircle[fill]{\pgfxy(20.00,20.00)}{1.00mm}
\pgfcircle[stroke]{\pgfxy(20.00,20.00)}{1.00mm}
\pgfcircle[fill]{\pgfxy(40.00,30.00)}{1.00mm}
\pgfcircle[stroke]{\pgfxy(40.00,30.00)}{1.00mm}
\pgfcircle[fill]{\pgfxy(50.00,40.00)}{1.00mm}
\pgfcircle[stroke]{\pgfxy(50.00,40.00)}{1.00mm}
\pgfcircle[fill]{\pgfxy(50.00,20.00)}{1.00mm}
\pgfcircle[stroke]{\pgfxy(50.00,20.00)}{1.00mm}
\pgfmoveto{\pgfxy(24.02,24.02)}\pgflineto{\pgfxy(20.00,20.00)}\pgfstroke
\pgfmoveto{\pgfxy(30.00,30.00)}\pgflineto{\pgfxy(24.02,24.02)}\pgfstroke
\pgfmoveto{\pgfxy(24.02,24.02)}\pgflineto{\pgfxy(26.50,25.51)}\pgflineto{\pgfxy(24.02,24.02)}\pgflineto{\pgfxy(25.51,26.50)}\pgflineto{\pgfxy(24.02,24.02)}\pgfclosepath\pgffill
\pgfmoveto{\pgfxy(24.02,24.02)}\pgflineto{\pgfxy(26.50,25.51)}\pgflineto{\pgfxy(24.02,24.02)}\pgflineto{\pgfxy(25.51,26.50)}\pgflineto{\pgfxy(24.02,24.02)}\pgfclosepath\pgfstroke
\pgfmoveto{\pgfxy(36.04,30.00)}\pgflineto{\pgfxy(40.00,30.00)}\pgfstroke
\pgfmoveto{\pgfxy(30.00,30.00)}\pgflineto{\pgfxy(36.00,30.00)}\pgfstroke
\pgfmoveto{\pgfxy(36.00,30.00)}\pgflineto{\pgfxy(33.20,30.70)}\pgflineto{\pgfxy(36.00,30.00)}\pgflineto{\pgfxy(33.20,29.30)}\pgflineto{\pgfxy(36.00,30.00)}\pgfclosepath\pgffill
\pgfmoveto{\pgfxy(36.00,30.00)}\pgflineto{\pgfxy(33.20,30.70)}\pgflineto{\pgfxy(36.00,30.00)}\pgflineto{\pgfxy(33.20,29.30)}\pgflineto{\pgfxy(36.00,30.00)}\pgfclosepath\pgfstroke
\pgfmoveto{\pgfxy(43.96,33.96)}\pgflineto{\pgfxy(40.00,30.00)}\pgfstroke
\pgfmoveto{\pgfxy(50.00,40.00)}\pgflineto{\pgfxy(44.00,34.00)}\pgfstroke
\pgfmoveto{\pgfxy(44.00,34.00)}\pgflineto{\pgfxy(46.47,35.48)}\pgflineto{\pgfxy(44.00,34.00)}\pgflineto{\pgfxy(45.48,36.47)}\pgflineto{\pgfxy(44.00,34.00)}\pgfclosepath\pgffill
\pgfmoveto{\pgfxy(44.00,34.00)}\pgflineto{\pgfxy(46.47,35.48)}\pgflineto{\pgfxy(44.00,34.00)}\pgflineto{\pgfxy(45.48,36.47)}\pgflineto{\pgfxy(44.00,34.00)}\pgfclosepath\pgfstroke
\pgfmoveto{\pgfxy(45.84,24.16)}\pgflineto{\pgfxy(50.00,20.00)}\pgfstroke
\pgfmoveto{\pgfxy(40.00,30.00)}\pgflineto{\pgfxy(45.84,24.16)}\pgfstroke
\pgfmoveto{\pgfxy(45.84,24.16)}\pgflineto{\pgfxy(44.36,26.63)}\pgflineto{\pgfxy(45.84,24.16)}\pgflineto{\pgfxy(43.37,25.64)}\pgflineto{\pgfxy(45.84,24.16)}\pgfclosepath\pgffill
\pgfmoveto{\pgfxy(45.84,24.16)}\pgflineto{\pgfxy(44.36,26.63)}\pgflineto{\pgfxy(45.84,24.16)}\pgflineto{\pgfxy(43.37,25.64)}\pgflineto{\pgfxy(45.84,24.16)}\pgfclosepath\pgfstroke
\pgfcircle[fill]{\pgfxy(85.00,40.00)}{1.00mm}
\pgfcircle[stroke]{\pgfxy(85.00,40.00)}{1.00mm}
\pgfcircle[fill]{\pgfxy(70.00,30.00)}{1.00mm}
\pgfcircle[stroke]{\pgfxy(70.00,30.00)}{1.00mm}
\pgfcircle[fill]{\pgfxy(85.00,30.00)}{1.00mm}
\pgfcircle[stroke]{\pgfxy(85.00,30.00)}{1.00mm}
\pgfcircle[fill]{\pgfxy(100.00,30.00)}{1.00mm}
\pgfcircle[stroke]{\pgfxy(100.00,30.00)}{1.00mm}
\pgfcircle[fill]{\pgfxy(70.00,20.00)}{1.00mm}
\pgfcircle[stroke]{\pgfxy(70.00,20.00)}{1.00mm}
\pgfcircle[fill]{\pgfxy(85.00,20.00)}{1.00mm}
\pgfcircle[stroke]{\pgfxy(85.00,20.00)}{1.00mm}
\pgfputat{\pgfxy(17.00,40.00)}{\pgfbox[bottom,left]{\fontsize{11.38}{13.66}\selectfont \makebox(0,0)[r]{1\strut}}}
\pgfputat{\pgfxy(53.00,40.00)}{\pgfbox[bottom,left]{\fontsize{11.38}{13.66}\selectfont \makebox(0,0)[l]{3\strut}}}
\pgfputat{\pgfxy(27.00,30.00)}{\pgfbox[bottom,left]{\fontsize{11.38}{13.66}\selectfont \makebox(0,0)[r]{2\strut}}}
\pgfputat{\pgfxy(43.00,30.00)}{\pgfbox[bottom,left]{\fontsize{11.38}{13.66}\selectfont \makebox(0,0)[l]{4\strut}}}
\pgfputat{\pgfxy(17.00,20.00)}{\pgfbox[bottom,left]{\fontsize{11.38}{13.66}\selectfont \makebox(0,0)[r]{5\strut}}}
\pgfputat{\pgfxy(53.00,20.00)}{\pgfbox[bottom,left]{\fontsize{11.38}{13.66}\selectfont \makebox(0,0)[l]{6\strut}}}
\pgfputat{\pgfxy(88.00,40.00)}{\pgfbox[bottom,left]{\fontsize{11.38}{13.66}\selectfont \makebox(0,0)[l]{4\strut}}}
\pgfputat{\pgfxy(73.00,30.00)}{\pgfbox[bottom,left]{\fontsize{11.38}{13.66}\selectfont \makebox(0,0)[l]{2\strut}}}
\pgfputat{\pgfxy(88.00,30.00)}{\pgfbox[bottom,left]{\fontsize{11.38}{13.66}\selectfont \makebox(0,0)[l]{3\strut}}}
\pgfputat{\pgfxy(103.00,30.00)}{\pgfbox[bottom,left]{\fontsize{11.38}{13.66}\selectfont \makebox(0,0)[l]{6\strut}}}
\pgfputat{\pgfxy(73.00,20.00)}{\pgfbox[bottom,left]{\fontsize{11.38}{13.66}\selectfont \makebox(0,0)[l]{1\strut}}}
\pgfputat{\pgfxy(88.00,20.00)}{\pgfbox[bottom,left]{\fontsize{11.38}{13.66}\selectfont \makebox(0,0)[l]{5\strut}}}
\pgfmoveto{\pgfxy(70.00,26.00)}\pgflineto{\pgfxy(70.00,30.00)}\pgfstroke
\pgfmoveto{\pgfxy(70.00,20.00)}\pgflineto{\pgfxy(70.00,26.39)}\pgfstroke
\pgfmoveto{\pgfxy(70.00,26.39)}\pgflineto{\pgfxy(69.30,23.59)}\pgflineto{\pgfxy(70.00,26.39)}\pgflineto{\pgfxy(70.70,23.59)}\pgflineto{\pgfxy(70.00,26.39)}\pgfclosepath\pgffill
\pgfmoveto{\pgfxy(70.00,26.39)}\pgflineto{\pgfxy(69.30,23.59)}\pgflineto{\pgfxy(70.00,26.39)}\pgflineto{\pgfxy(70.70,23.59)}\pgflineto{\pgfxy(70.00,26.39)}\pgfclosepath\pgfstroke
\pgfmoveto{\pgfxy(85.00,36.00)}\pgflineto{\pgfxy(85.00,40.00)}\pgfstroke
\pgfmoveto{\pgfxy(85.00,30.00)}\pgflineto{\pgfxy(85.00,35.80)}\pgfstroke
\pgfmoveto{\pgfxy(85.00,35.80)}\pgflineto{\pgfxy(84.30,33.00)}\pgflineto{\pgfxy(85.00,35.80)}\pgflineto{\pgfxy(85.70,33.00)}\pgflineto{\pgfxy(85.00,35.80)}\pgfclosepath\pgffill
\pgfmoveto{\pgfxy(85.00,35.80)}\pgflineto{\pgfxy(84.30,33.00)}\pgflineto{\pgfxy(85.00,35.80)}\pgflineto{\pgfxy(85.70,33.00)}\pgflineto{\pgfxy(85.00,35.80)}\pgfclosepath\pgfstroke
\pgfmoveto{\pgfxy(78.84,35.89)}\pgflineto{\pgfxy(85.00,40.00)}\pgfstroke
\pgfmoveto{\pgfxy(70.00,30.00)}\pgflineto{\pgfxy(78.84,35.89)}\pgfstroke
\pgfmoveto{\pgfxy(78.84,35.89)}\pgflineto{\pgfxy(76.12,34.92)}\pgflineto{\pgfxy(78.84,35.89)}\pgflineto{\pgfxy(76.90,33.76)}\pgflineto{\pgfxy(78.84,35.89)}\pgfclosepath\pgffill
\pgfmoveto{\pgfxy(78.84,35.89)}\pgflineto{\pgfxy(76.12,34.92)}\pgflineto{\pgfxy(78.84,35.89)}\pgflineto{\pgfxy(76.90,33.76)}\pgflineto{\pgfxy(78.84,35.89)}\pgfclosepath\pgfstroke
\pgfmoveto{\pgfxy(90.54,36.31)}\pgflineto{\pgfxy(85.00,40.00)}\pgfstroke
\pgfmoveto{\pgfxy(100.00,30.00)}\pgflineto{\pgfxy(90.54,36.31)}\pgfstroke
\pgfmoveto{\pgfxy(90.54,36.31)}\pgflineto{\pgfxy(92.48,34.17)}\pgflineto{\pgfxy(90.54,36.31)}\pgflineto{\pgfxy(93.25,35.34)}\pgflineto{\pgfxy(90.54,36.31)}\pgfclosepath\pgffill
\pgfmoveto{\pgfxy(90.54,36.31)}\pgflineto{\pgfxy(92.48,34.17)}\pgflineto{\pgfxy(90.54,36.31)}\pgflineto{\pgfxy(93.25,35.34)}\pgflineto{\pgfxy(90.54,36.31)}\pgfclosepath\pgfstroke
\pgfputat{\pgfxy(35.00,13.00)}{\pgfbox[bottom,left]{\fontsize{11.38}{13.66}\selectfont \makebox[0pt]{$\lambda=1^12^2$}}}
\pgfputat{\pgfxy(85.00,13.00)}{\pgfbox[bottom,left]{\fontsize{11.38}{13.66}\selectfont \makebox[0pt]{$\lambda=2^13^1$}}}
\pgfmoveto{\pgfxy(75.69,26.20)}\pgflineto{\pgfxy(70.00,30.00)}\pgfstroke
\pgfmoveto{\pgfxy(85.00,20.00)}\pgflineto{\pgfxy(75.69,26.20)}\pgfstroke
\pgfmoveto{\pgfxy(75.69,26.20)}\pgflineto{\pgfxy(77.64,24.07)}\pgflineto{\pgfxy(75.69,26.20)}\pgflineto{\pgfxy(78.41,25.23)}\pgflineto{\pgfxy(75.69,26.20)}\pgfclosepath\pgffill
\pgfmoveto{\pgfxy(75.69,26.20)}\pgflineto{\pgfxy(77.64,24.07)}\pgflineto{\pgfxy(75.69,26.20)}\pgflineto{\pgfxy(78.41,25.23)}\pgflineto{\pgfxy(75.69,26.20)}\pgfclosepath\pgfstroke
\pgfputat{\pgfxy(35.00,44.00)}{\pgfbox[bottom,left]{\fontsize{11.38}{13.66}\selectfont \makebox[0pt]{local orientation}}}
\pgfputat{\pgfxy(85.00,44.00)}{\pgfbox[bottom,left]{\fontsize{11.38}{13.66}\selectfont \makebox[0pt]{global orientation}}}
\pgfcircle[stroke]{\pgfxy(85.00,40.00)}{2.00mm}
\end{pgfpicture}%
\caption{local and global indegree sequences}
\label{label}
\end{figure}

For any partition $\lambda$ of $n-1$ and $r\in [n]$, let $\Tnl$ (resp. $\Tnlr$) be the subset of trees in $\Tn$ (resp. $\Tnr$) with local (resp. global)
indegree sequence $\lambda$.

The problem of counting the trees with a given
indegree sequence was first encountered by Cotterill in his study of algebraic geometry. In particular, Cotterill~\cite[Eq. (3.34)]{2007arXiv0706.2049C} made the following conjecture.
\begin{conj}\label{thm:first}
Let $\lambda = 1^{e_1} 2^{e_2} \ldots$ be a partition of $n-1$ and $e_0 = n-\ell(\lambda)$. Then
 the cardinality of $\Tnl$ equals
\begin{equation}\label{first}
 \dfrac{(n-1)!^2}{e_0!(0!)^{e_0} e_1! (1!)^{e_1} e_2! (2!)^{e_2} \ldots}.
\end{equation}
\end{conj}

This remarkable formula is reminiscent  to at least two known enumerative problems.
The {\em type} of a set-partition $\pi$ is the integer partition $1^{e_1} 2^{e_2} \ldots$ if $e_h$ blocks of $\pi$ have size $i$, we denote it by $\type(\pi)$.
Let $\Pi_{n,\lambda}$ be the set of partitions of a $(n-1)$-element
 set of {\em type} $\lambda=1^{e_1} 2^{e_2} \ldots$.   Since the cardinality  of $\Pi_{n,\lambda}$ is easily seen  to equal
$ (n-1)!/{e_1! (1!)^{e_1} e_2! (2!)^{e_2} \ldots}$,
Stanley (see \cite{DY10}) noticed  that
the  formula \eqref{first} can be written as
$\abs{\Pi_{n,\lambda}}\cdot \frac{(n-1)!}{(n-\ell(\lambda))!}$.
Based on this factorization   a proof of Conjecture ~\ref{thm:first} was given by Du and Yin \cite{DY10}
by using M\"obius inversion formula on the poset of set partitions.
Obviously a bijective proof of this result  is highly desired.
More precisely,
for $k\in [n]$,
a {\em $k$-permutation} of $[n]$ is an ordered sequence of $k$ elements selected from $[n]$, without repetitions.
Denote by $\Snkr$ the set of $k$-permutations $(p_{1}, \ldots, p_k)$ of $[n]$ with $p_k=r$.
The cardinality of $\Snkr$ is equal to
$(n-1)\ldots(n-k+1) = (n-1)!/{(n-k)!}$.
It follows that a bijection between $\Tnl$ and
$\Pi_{n,\lambda}\times {\mathcal S}_{n,\ell(\lambda)}^{(r)}$ will give
 a bijective proof of Conjecture~1. We shall construct such a bijection via labeled rooted trees. Indeed, for a given partition $\lambda=1^{e_1} 2^{e_2} \ldots$ of $n-1$,
the cardinality of $\Tnlr$ is independent of the choice $r\in [n]$.
From the known formula for the total number of rooted trees on $[n]$ with global indegree sequence of type $\lambda$ (see, for example, \cite[Corollary 5.3.5]{MR1676282}) we derive
that   the cardinality  of $\Tnlr$   is given by
 \eqref{first}.  For our purpose, we will first exhibit a   \emph{Pr\"ufer-like code} for rooted trees to prove this result.
 %%%%%%%%%%%%%%%%
\begin{thm}\label{thm:prufer}
Let $\lambda = 1^{e_1} 2^{e_2} \ldots$ be a partition of $n-1$  and $r\in [n]$. There is a bijection between $\Tnlr$ and $\Pi_{n,\lambda}\times {\mathcal S}_{n,\ell(\lambda)}^{(r)}$.
\end{thm}

Therefore, Cotterill's conjecture will be proved if we can establish
a bijection from (unrooted) trees to rooted trees
such that the local indegree sequence of a (unrooted) tree equals the global indegree sequence of the corresponding rooted tree.
%The main purpose of this  paper is to give such a bijection.
The following is our second main theorem.
%%%%%%%
\begin{thm}\label{thm:main}For any $r\in [n]$,
there is  a bijection $\Phi_r: \Tnl\to \Tnlr$.
\end{thm}

Besides, Cotterill~\cite[Eq. (3.39)]{2007arXiv0706.2049C} also conjectured the following formula:
\begin{equation}\label{second}
\sum_{|\lambda|=n-1 \atop e_0+e_1+\cdots =n} \dfrac{(n-1)!}{e_0!e_1!e_2! \ldots } \sum_{i\ge 0} e_h {i + 1 \choose 2} = {2n-1 \choose n-2}.
\end{equation}
In a previous  version of this paper, we  proved
\begin{equation}\label{general}
\sum_{|\lambda|=m-1 \atop e_0+e_1+\cdots =n} {n \choose e_o, e_1, e_2, \ldots } \sum_{i\ge 0} e_h { i+p-l \choose p} = n {n+m-2+p-l \choose n-1+p}.
\end{equation}
and pointed out   that  \eqref{second}  is the $m=n$, $p=2$, and $l=1$ case of
\eqref{general}. After submitting the paper, Ole Warnaar (Personal communication) kindly conveyed us with his believe that
a $q$-analogue of \eqref{general} must exist and sent us an identity  on the Hall-Littlewood functions in the spirit of \cite{MR2200851}.
Our third aim is to  present  the  $q$-analogue of \eqref{general} derived  from  Warnaar's original identity.
For any partition $\lambda$, let $\lambda'=(\lambda'_1,\lambda'_2,\ldots)$ be its conjugate and $n(\lambda)=\sum_i {\lambda'_i \choose 2}$.
Note that $\ell(\lambda)=\lambda'_1$.
Introduce the $q$-shifted factorial:
$$
(a)_k:=(a;q)_k=(1-a)(1-aq)\cdots (1-aq^{k-1})\quad \textrm{ for $k\geq 0$}.
$$
The $q$-binomial and $q$-multinomial coefficients are defined
by
$$
{n\brack k}_q=\frac{(q;q)_n}{(q;q)_k(q;q)_{n-k}}\quad\textrm{and}\quad
{n\brack e_0,e_1,\ldots,e_l}_q=\frac{(q;q)_n}{(q;q)_{e_0} (q;q)_{e_1}\cdots (q;q)_{e_l}},
$$
where $e_0+\cdots +e_l=n$.

\begin{thm}\label{thm:qbinomial}
For nonnegative positive integers $m$, $n$, $l$ and $p$ such that $m,n\geq 1$, there holds
\begin{align}
\sum_{|\lambda|=m-1, \ell(\lambda)\leq n} &q^{(p+1)(m-1)+2 n(\lambda)}{n\brack e_0,e_1,\ldots}_q\nonumber\\
&\times
\sum_{i\geq 0}q^{(1-p)i-2\sum_{k=1}^i \lambda_k'}
\qbin{i+p-l}{p}[e_h]_q=[n]_q{n+m-2+p-l\brack n-1+p}_q,
\label{eq:gen}
\end{align}
where $e_h= \lambda'_i - \lambda'_{i+1}$ with  $\lambda'_0 = n$.
\end{thm}

This paper is organized as follows: In Section~\ref{bi}, we give a Pr\"ufer-like code for rooted labeled trees to prove Theorem~\ref{bi},  and in Section~\ref{mainsection}, we prove Theorem~\ref{thm:main} by constructing a bijection from unrooted labeled trees to rooted labeled trees, which maps local indegree sequence to global indegree sequence. In Section~\ref{four}, we prove Theorem~\ref{thm:qbinomial}.
In the last section, we discuss a connection between Remmel and Williamson's generating function~\cite{MR1928786} for trees with respect to the
indegree type
 and Coterill's formula \eqref{first}.

We close this section with some further definitions.
Throughout this paper, we denote by $\typeloc(T)$ (resp. $\typeglo(T)$) the local (resp. global) indegree sequence of a tree $T$ as an integer partition.
%Let $\Pinr$ be the set of partitions of the set $[n]\setminus \set{r}$ and
Let $\Pinkr$ be the set of partitions of the set $[n]\setminus \set{r}$ with $k$ parts.
%%%%%%%%%%%%%%%%%%%%%%%%%%%%%%%%
%\section{A Pr\"ufer-like code for rooted trees}\label{bi}
\section{Proof of Theorem \ref{thm:prufer}}\label{bi}
%%%%%%%%%%%%%%%%%%%%%%%%%%%%%%%%
The classical {\em Pr\"ufer code} for a rooted tree is the sequence obtained
by cutting recursively the largest {\em leave} and recording its parent (see \cite[P.25]{MR1676282}).
In this section, we shall give an analogous code for rooted trees by replacing leaves by {\em leaf-groups}.

Given a rooted tree $T$, a vertex $v$ of $T$ is called a {\em leaf} if the global indegree of $v$ is 0.
If $i \to j$ is an edge of $T$, then $i$ (resp. $j$) is called the {\em child} (resp. {\em parent}) of $j$ (resp. $i$).
The set of all the children of $v$ is called its {\em child-group}, denoted by $G_v$.
In particular, a child-group is called {\em leaf-group} if all the children are leaves.
Moreover, we order the leaf-groups by their maximal elements. For example, we have
\begin{align}\label{eq:order}
\{5,9,12\}>\{2,11\}.
\end{align}

\begin{figure}[t]
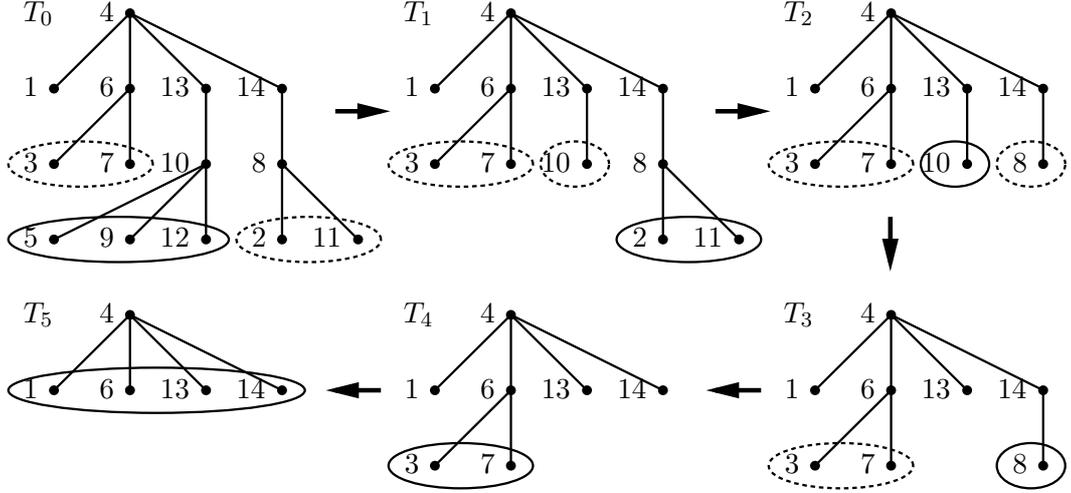

\centering
\begin{pgfpicture}{22.00mm}{-25.00mm}{165.00mm}{44.14mm}
\pgfsetxvec{\pgfpoint{1.00mm}{0mm}}
\pgfsetyvec{\pgfpoint{0mm}{1.00mm}}
\color[rgb]{0,0,0}\pgfsetlinewidth{0.30mm}\pgfsetdash{}{0mm}
\pgfputat{\pgfxy(36.00,39.00)}{\pgfbox[bottom,left]{\fontsize{11.38}{13.66}\selectfont 4}}
\pgfputat{\pgfxy(26.00,29.00)}{\pgfbox[bottom,left]{\fontsize{11.38}{13.66}\selectfont 1}}
\pgfcircle[fill]{\pgfxy(40.00,40.00)}{0.50mm}
\pgfcircle[stroke]{\pgfxy(40.00,40.00)}{0.50mm}
\pgfcircle[fill]{\pgfxy(30.00,30.00)}{0.50mm}
\pgfcircle[stroke]{\pgfxy(30.00,30.00)}{0.50mm}
\pgfcircle[fill]{\pgfxy(40.00,30.00)}{0.50mm}
\pgfcircle[stroke]{\pgfxy(40.00,30.00)}{0.50mm}
\pgfcircle[fill]{\pgfxy(50.00,30.00)}{0.50mm}
\pgfcircle[stroke]{\pgfxy(50.00,30.00)}{0.50mm}
\pgfcircle[fill]{\pgfxy(60.00,30.00)}{0.50mm}
\pgfcircle[stroke]{\pgfxy(60.00,30.00)}{0.50mm}
\pgfcircle[fill]{\pgfxy(30.00,20.00)}{0.50mm}
\pgfcircle[stroke]{\pgfxy(30.00,20.00)}{0.50mm}
\pgfcircle[fill]{\pgfxy(40.00,20.00)}{0.50mm}
\pgfcircle[stroke]{\pgfxy(40.00,20.00)}{0.50mm}
\pgfcircle[fill]{\pgfxy(50.00,20.00)}{0.50mm}
\pgfcircle[stroke]{\pgfxy(50.00,20.00)}{0.50mm}
\pgfcircle[fill]{\pgfxy(60.00,20.00)}{0.50mm}
\pgfcircle[stroke]{\pgfxy(60.00,20.00)}{0.50mm}
\pgfcircle[fill]{\pgfxy(60.00,10.00)}{0.50mm}
\pgfcircle[stroke]{\pgfxy(60.00,10.00)}{0.50mm}
\pgfcircle[fill]{\pgfxy(70.00,10.00)}{0.50mm}
\pgfcircle[stroke]{\pgfxy(70.00,10.00)}{0.50mm}
\pgfmoveto{\pgfxy(40.00,40.00)}\pgflineto{\pgfxy(30.00,30.00)}\pgfstroke
\pgfmoveto{\pgfxy(40.00,40.00)}\pgflineto{\pgfxy(40.00,30.00)}\pgfstroke
\pgfmoveto{\pgfxy(40.00,40.00)}\pgflineto{\pgfxy(50.00,30.00)}\pgfstroke
\pgfmoveto{\pgfxy(40.00,40.00)}\pgflineto{\pgfxy(60.00,30.00)}\pgfstroke
\pgfmoveto{\pgfxy(60.00,30.00)}\pgflineto{\pgfxy(60.00,20.00)}\pgfstroke
\pgfmoveto{\pgfxy(40.00,30.00)}\pgflineto{\pgfxy(30.00,20.00)}\pgfstroke
\pgfmoveto{\pgfxy(40.00,30.00)}\pgflineto{\pgfxy(40.00,20.00)}\pgfstroke
\pgfmoveto{\pgfxy(50.00,30.00)}\pgflineto{\pgfxy(50.00,20.00)}\pgfstroke
\pgfmoveto{\pgfxy(60.00,20.00)}\pgflineto{\pgfxy(60.00,10.00)}\pgfstroke
\pgfmoveto{\pgfxy(60.00,20.00)}\pgflineto{\pgfxy(70.00,10.00)}\pgfstroke
\pgfputat{\pgfxy(56.00,9.00)}{\pgfbox[bottom,left]{\fontsize{11.38}{13.66}\selectfont 2}}
\pgfputat{\pgfxy(36.00,29.00)}{\pgfbox[bottom,left]{\fontsize{11.38}{13.66}\selectfont 6}}
\pgfputat{\pgfxy(56.00,19.00)}{\pgfbox[bottom,left]{\fontsize{11.38}{13.66}\selectfont 8}}
\pgfputat{\pgfxy(26.00,19.00)}{\pgfbox[bottom,left]{\fontsize{11.38}{13.66}\selectfont 3}}
\pgfputat{\pgfxy(36.00,19.00)}{\pgfbox[bottom,left]{\fontsize{11.38}{13.66}\selectfont 7}}
\pgfputat{\pgfxy(64.00,9.00)}{\pgfbox[bottom,left]{\fontsize{11.38}{13.66}\selectfont 11}}
\pgfputat{\pgfxy(44.00,29.00)}{\pgfbox[bottom,left]{\fontsize{11.38}{13.66}\selectfont 13}}
\pgfputat{\pgfxy(44.00,19.00)}{\pgfbox[bottom,left]{\fontsize{11.38}{13.66}\selectfont 10}}
\pgfputat{\pgfxy(54.00,29.00)}{\pgfbox[bottom,left]{\fontsize{11.38}{13.66}\selectfont 14}}
\pgfsetlinewidth{0.60mm}\pgfmoveto{\pgfxy(67.00,27.00)}\pgflineto{\pgfxy(73.00,27.00)}\pgfstroke
\pgfmoveto{\pgfxy(73.00,27.00)}\pgflineto{\pgfxy(70.20,27.70)}\pgflineto{\pgfxy(70.20,26.30)}\pgflineto{\pgfxy(73.00,27.00)}\pgfclosepath\pgffill
\pgfmoveto{\pgfxy(73.00,27.00)}\pgflineto{\pgfxy(70.20,27.70)}\pgflineto{\pgfxy(70.20,26.30)}\pgflineto{\pgfxy(73.00,27.00)}\pgfclosepath\pgfstroke
\pgfmoveto{\pgfxy(117.00,27.00)}\pgflineto{\pgfxy(123.00,27.00)}\pgfstroke
\pgfmoveto{\pgfxy(123.00,27.00)}\pgflineto{\pgfxy(120.20,27.70)}\pgflineto{\pgfxy(120.20,26.30)}\pgflineto{\pgfxy(123.00,27.00)}\pgfclosepath\pgffill
\pgfmoveto{\pgfxy(123.00,27.00)}\pgflineto{\pgfxy(120.20,27.70)}\pgflineto{\pgfxy(120.20,26.30)}\pgflineto{\pgfxy(123.00,27.00)}\pgfclosepath\pgfstroke
\pgfmoveto{\pgfxy(140.00,13.00)}\pgflineto{\pgfxy(140.00,7.00)}\pgfstroke
\pgfmoveto{\pgfxy(140.00,7.00)}\pgflineto{\pgfxy(140.70,9.80)}\pgflineto{\pgfxy(139.30,9.80)}\pgflineto{\pgfxy(140.00,7.00)}\pgfclosepath\pgffill
\pgfmoveto{\pgfxy(140.00,7.00)}\pgflineto{\pgfxy(140.70,9.80)}\pgflineto{\pgfxy(139.30,9.80)}\pgflineto{\pgfxy(140.00,7.00)}\pgfclosepath\pgfstroke
\pgfmoveto{\pgfxy(123.00,-10.00)}\pgflineto{\pgfxy(117.00,-10.00)}\pgfstroke
\pgfmoveto{\pgfxy(117.00,-10.00)}\pgflineto{\pgfxy(119.80,-10.70)}\pgflineto{\pgfxy(119.80,-9.30)}\pgflineto{\pgfxy(117.00,-10.00)}\pgfclosepath\pgffill
\pgfmoveto{\pgfxy(117.00,-10.00)}\pgflineto{\pgfxy(119.80,-10.70)}\pgflineto{\pgfxy(119.80,-9.30)}\pgflineto{\pgfxy(117.00,-10.00)}\pgfclosepath\pgfstroke
\pgfmoveto{\pgfxy(73.00,-10.00)}\pgflineto{\pgfxy(67.00,-10.00)}\pgfstroke
\pgfmoveto{\pgfxy(67.00,-10.00)}\pgflineto{\pgfxy(69.80,-10.70)}\pgflineto{\pgfxy(69.80,-9.30)}\pgflineto{\pgfxy(67.00,-10.00)}\pgfclosepath\pgffill
\pgfmoveto{\pgfxy(67.00,-10.00)}\pgflineto{\pgfxy(69.80,-10.70)}\pgflineto{\pgfxy(69.80,-9.30)}\pgflineto{\pgfxy(67.00,-10.00)}\pgfclosepath\pgfstroke
\pgfputat{\pgfxy(26.00,39.00)}{\pgfbox[bottom,left]{\fontsize{11.38}{13.66}\selectfont $T_0$}}
\pgfputat{\pgfxy(76.00,39.00)}{\pgfbox[bottom,left]{\fontsize{11.38}{13.66}\selectfont $T_1$}}
\pgfputat{\pgfxy(126.00,39.00)}{\pgfbox[bottom,left]{\fontsize{11.38}{13.66}\selectfont $T_2$}}
\pgfputat{\pgfxy(126.00,-1.00)}{\pgfbox[bottom,left]{\fontsize{11.38}{13.66}\selectfont $T_3$}}
\pgfputat{\pgfxy(76.00,-1.00)}{\pgfbox[bottom,left]{\fontsize{11.38}{13.66}\selectfont $T_4$}}
\pgfputat{\pgfxy(26.00,-1.00)}{\pgfbox[bottom,left]{\fontsize{11.38}{13.66}\selectfont $T_5$}}
\pgfsetlinewidth{0.30mm}\pgfmoveto{\pgfxy(50.00,20.00)}\pgflineto{\pgfxy(40.00,10.00)}\pgfstroke
\pgfmoveto{\pgfxy(50.00,20.00)}\pgflineto{\pgfxy(30.00,10.00)}\pgfstroke
\pgfcircle[fill]{\pgfxy(30.00,10.00)}{0.50mm}
\pgfcircle[stroke]{\pgfxy(30.00,10.00)}{0.50mm}
\pgfcircle[fill]{\pgfxy(40.00,10.00)}{0.50mm}
\pgfcircle[stroke]{\pgfxy(40.00,10.00)}{0.50mm}
\pgfmoveto{\pgfxy(50.00,20.00)}\pgflineto{\pgfxy(50.00,10.00)}\pgfstroke
\pgfcircle[fill]{\pgfxy(50.00,10.00)}{0.50mm}
\pgfcircle[stroke]{\pgfxy(50.00,10.00)}{0.50mm}
\pgfputat{\pgfxy(44.00,9.00)}{\pgfbox[bottom,left]{\fontsize{11.38}{13.66}\selectfont 12}}
\pgfputat{\pgfxy(36.00,9.00)}{\pgfbox[bottom,left]{\fontsize{11.38}{13.66}\selectfont 9}}
\pgfputat{\pgfxy(26.00,9.00)}{\pgfbox[bottom,left]{\fontsize{11.38}{13.66}\selectfont 5}}
\pgfputat{\pgfxy(86.10,39.00)}{\pgfbox[bottom,left]{\fontsize{11.38}{13.66}\selectfont 4}}
\pgfputat{\pgfxy(76.10,29.00)}{\pgfbox[bottom,left]{\fontsize{11.38}{13.66}\selectfont 1}}
\pgfcircle[fill]{\pgfxy(90.10,40.00)}{0.50mm}
\pgfcircle[stroke]{\pgfxy(90.10,40.00)}{0.50mm}
\pgfcircle[fill]{\pgfxy(80.10,30.00)}{0.50mm}
\pgfcircle[stroke]{\pgfxy(80.10,30.00)}{0.50mm}
\pgfcircle[fill]{\pgfxy(90.10,30.00)}{0.50mm}
\pgfcircle[stroke]{\pgfxy(90.10,30.00)}{0.50mm}
\pgfcircle[fill]{\pgfxy(100.10,30.00)}{0.50mm}
\pgfcircle[stroke]{\pgfxy(100.10,30.00)}{0.50mm}
\pgfcircle[fill]{\pgfxy(110.10,30.00)}{0.50mm}
\pgfcircle[stroke]{\pgfxy(110.10,30.00)}{0.50mm}
\pgfcircle[fill]{\pgfxy(80.10,20.00)}{0.50mm}
\pgfcircle[stroke]{\pgfxy(80.10,20.00)}{0.50mm}
\pgfcircle[fill]{\pgfxy(90.10,20.00)}{0.50mm}
\pgfcircle[stroke]{\pgfxy(90.10,20.00)}{0.50mm}
\pgfcircle[fill]{\pgfxy(100.10,20.00)}{0.50mm}
\pgfcircle[stroke]{\pgfxy(100.10,20.00)}{0.50mm}
\pgfcircle[fill]{\pgfxy(110.10,20.00)}{0.50mm}
\pgfcircle[stroke]{\pgfxy(110.10,20.00)}{0.50mm}
\pgfcircle[fill]{\pgfxy(110.10,10.00)}{0.50mm}
\pgfcircle[stroke]{\pgfxy(110.10,10.00)}{0.50mm}
\pgfcircle[fill]{\pgfxy(120.10,10.00)}{0.50mm}
\pgfcircle[stroke]{\pgfxy(120.10,10.00)}{0.50mm}
\pgfmoveto{\pgfxy(90.10,40.00)}\pgflineto{\pgfxy(80.10,30.00)}\pgfstroke
\pgfmoveto{\pgfxy(90.10,40.00)}\pgflineto{\pgfxy(90.10,30.00)}\pgfstroke
\pgfmoveto{\pgfxy(90.10,40.00)}\pgflineto{\pgfxy(100.10,30.00)}\pgfstroke
\pgfmoveto{\pgfxy(90.10,40.00)}\pgflineto{\pgfxy(110.10,30.00)}\pgfstroke
\pgfmoveto{\pgfxy(110.10,30.00)}\pgflineto{\pgfxy(110.10,20.00)}\pgfstroke
\pgfmoveto{\pgfxy(90.10,30.00)}\pgflineto{\pgfxy(80.10,20.00)}\pgfstroke
\pgfmoveto{\pgfxy(90.10,30.00)}\pgflineto{\pgfxy(90.10,20.00)}\pgfstroke
\pgfmoveto{\pgfxy(100.10,30.00)}\pgflineto{\pgfxy(100.10,20.00)}\pgfstroke
\pgfmoveto{\pgfxy(110.10,20.00)}\pgflineto{\pgfxy(110.10,10.00)}\pgfstroke
\pgfmoveto{\pgfxy(110.10,20.00)}\pgflineto{\pgfxy(120.10,10.00)}\pgfstroke
\pgfputat{\pgfxy(106.10,9.00)}{\pgfbox[bottom,left]{\fontsize{11.38}{13.66}\selectfont 2}}
\pgfputat{\pgfxy(86.10,29.00)}{\pgfbox[bottom,left]{\fontsize{11.38}{13.66}\selectfont 6}}
\pgfputat{\pgfxy(106.10,19.00)}{\pgfbox[bottom,left]{\fontsize{11.38}{13.66}\selectfont 8}}
\pgfputat{\pgfxy(76.10,19.00)}{\pgfbox[bottom,left]{\fontsize{11.38}{13.66}\selectfont 3}}
\pgfputat{\pgfxy(86.10,19.00)}{\pgfbox[bottom,left]{\fontsize{11.38}{13.66}\selectfont 7}}
\pgfputat{\pgfxy(114.10,9.00)}{\pgfbox[bottom,left]{\fontsize{11.38}{13.66}\selectfont 11}}
\pgfputat{\pgfxy(94.10,29.00)}{\pgfbox[bottom,left]{\fontsize{11.38}{13.66}\selectfont 13}}
\pgfputat{\pgfxy(94.10,19.00)}{\pgfbox[bottom,left]{\fontsize{11.38}{13.66}\selectfont 10}}
\pgfputat{\pgfxy(104.10,29.00)}{\pgfbox[bottom,left]{\fontsize{11.38}{13.66}\selectfont 14}}
\pgfputat{\pgfxy(136.10,39.00)}{\pgfbox[bottom,left]{\fontsize{11.38}{13.66}\selectfont 4}}
\pgfputat{\pgfxy(126.10,29.00)}{\pgfbox[bottom,left]{\fontsize{11.38}{13.66}\selectfont 1}}
\pgfcircle[fill]{\pgfxy(140.10,40.00)}{0.50mm}
\pgfcircle[stroke]{\pgfxy(140.10,40.00)}{0.50mm}
\pgfcircle[fill]{\pgfxy(130.10,30.00)}{0.50mm}
\pgfcircle[stroke]{\pgfxy(130.10,30.00)}{0.50mm}
\pgfcircle[fill]{\pgfxy(140.10,30.00)}{0.50mm}
\pgfcircle[stroke]{\pgfxy(140.10,30.00)}{0.50mm}
\pgfcircle[fill]{\pgfxy(150.10,30.00)}{0.50mm}
\pgfcircle[stroke]{\pgfxy(150.10,30.00)}{0.50mm}
\pgfcircle[fill]{\pgfxy(160.10,30.00)}{0.50mm}
\pgfcircle[stroke]{\pgfxy(160.10,30.00)}{0.50mm}
\pgfcircle[fill]{\pgfxy(130.10,20.00)}{0.50mm}
\pgfcircle[stroke]{\pgfxy(130.10,20.00)}{0.50mm}
\pgfcircle[fill]{\pgfxy(140.10,20.00)}{0.50mm}
\pgfcircle[stroke]{\pgfxy(140.10,20.00)}{0.50mm}
\pgfcircle[fill]{\pgfxy(150.10,20.00)}{0.50mm}
\pgfcircle[stroke]{\pgfxy(150.10,20.00)}{0.50mm}
\pgfcircle[fill]{\pgfxy(160.10,20.00)}{0.50mm}
\pgfcircle[stroke]{\pgfxy(160.10,20.00)}{0.50mm}
\pgfmoveto{\pgfxy(140.10,40.00)}\pgflineto{\pgfxy(130.10,30.00)}\pgfstroke
\pgfmoveto{\pgfxy(140.10,40.00)}\pgflineto{\pgfxy(140.10,30.00)}\pgfstroke
\pgfmoveto{\pgfxy(140.10,40.00)}\pgflineto{\pgfxy(150.10,30.00)}\pgfstroke
\pgfmoveto{\pgfxy(140.10,40.00)}\pgflineto{\pgfxy(160.10,30.00)}\pgfstroke
\pgfmoveto{\pgfxy(160.10,30.00)}\pgflineto{\pgfxy(160.10,20.00)}\pgfstroke
\pgfmoveto{\pgfxy(140.10,30.00)}\pgflineto{\pgfxy(130.10,20.00)}\pgfstroke
\pgfmoveto{\pgfxy(140.10,30.00)}\pgflineto{\pgfxy(140.10,20.00)}\pgfstroke
\pgfmoveto{\pgfxy(150.10,30.00)}\pgflineto{\pgfxy(150.10,20.00)}\pgfstroke
\pgfputat{\pgfxy(136.10,29.00)}{\pgfbox[bottom,left]{\fontsize{11.38}{13.66}\selectfont 6}}
\pgfputat{\pgfxy(156.10,19.00)}{\pgfbox[bottom,left]{\fontsize{11.38}{13.66}\selectfont 8}}
\pgfputat{\pgfxy(126.10,19.00)}{\pgfbox[bottom,left]{\fontsize{11.38}{13.66}\selectfont 3}}
\pgfputat{\pgfxy(136.10,19.00)}{\pgfbox[bottom,left]{\fontsize{11.38}{13.66}\selectfont 7}}
\pgfputat{\pgfxy(144.10,29.00)}{\pgfbox[bottom,left]{\fontsize{11.38}{13.66}\selectfont 13}}
\pgfputat{\pgfxy(144.10,19.00)}{\pgfbox[bottom,left]{\fontsize{11.38}{13.66}\selectfont 10}}
\pgfputat{\pgfxy(154.10,29.00)}{\pgfbox[bottom,left]{\fontsize{11.38}{13.66}\selectfont 14}}
\pgfputat{\pgfxy(36.00,-1.00)}{\pgfbox[bottom,left]{\fontsize{11.38}{13.66}\selectfont 4}}
\pgfputat{\pgfxy(26.00,-11.00)}{\pgfbox[bottom,left]{\fontsize{11.38}{13.66}\selectfont 1}}
\pgfcircle[fill]{\pgfxy(40.00,0.00)}{0.50mm}
\pgfcircle[stroke]{\pgfxy(40.00,0.00)}{0.50mm}
\pgfcircle[fill]{\pgfxy(30.00,-10.00)}{0.50mm}
\pgfcircle[stroke]{\pgfxy(30.00,-10.00)}{0.50mm}
\pgfcircle[fill]{\pgfxy(40.00,-10.00)}{0.50mm}
\pgfcircle[stroke]{\pgfxy(40.00,-10.00)}{0.50mm}
\pgfcircle[fill]{\pgfxy(50.00,-10.00)}{0.50mm}
\pgfcircle[stroke]{\pgfxy(50.00,-10.00)}{0.50mm}
\pgfcircle[fill]{\pgfxy(60.00,-10.00)}{0.50mm}
\pgfcircle[stroke]{\pgfxy(60.00,-10.00)}{0.50mm}
\pgfmoveto{\pgfxy(40.00,0.00)}\pgflineto{\pgfxy(30.00,-10.00)}\pgfstroke
\pgfmoveto{\pgfxy(40.00,0.00)}\pgflineto{\pgfxy(40.00,-10.00)}\pgfstroke
\pgfmoveto{\pgfxy(40.00,0.00)}\pgflineto{\pgfxy(50.00,-10.00)}\pgfstroke
\pgfmoveto{\pgfxy(40.00,0.00)}\pgflineto{\pgfxy(60.00,-10.00)}\pgfstroke
\pgfputat{\pgfxy(36.00,-11.00)}{\pgfbox[bottom,left]{\fontsize{11.38}{13.66}\selectfont 6}}
\pgfputat{\pgfxy(44.00,-11.00)}{\pgfbox[bottom,left]{\fontsize{11.38}{13.66}\selectfont 13}}
\pgfputat{\pgfxy(54.00,-11.00)}{\pgfbox[bottom,left]{\fontsize{11.38}{13.66}\selectfont 14}}
\pgfputat{\pgfxy(86.10,-1.00)}{\pgfbox[bottom,left]{\fontsize{11.38}{13.66}\selectfont 4}}
\pgfputat{\pgfxy(76.10,-11.00)}{\pgfbox[bottom,left]{\fontsize{11.38}{13.66}\selectfont 1}}
\pgfcircle[fill]{\pgfxy(90.10,0.00)}{0.50mm}
\pgfcircle[stroke]{\pgfxy(90.10,0.00)}{0.50mm}
\pgfcircle[fill]{\pgfxy(80.10,-10.00)}{0.50mm}
\pgfcircle[stroke]{\pgfxy(80.10,-10.00)}{0.50mm}
\pgfcircle[fill]{\pgfxy(90.10,-10.00)}{0.50mm}
\pgfcircle[stroke]{\pgfxy(90.10,-10.00)}{0.50mm}
\pgfcircle[fill]{\pgfxy(100.10,-10.00)}{0.50mm}
\pgfcircle[stroke]{\pgfxy(100.10,-10.00)}{0.50mm}
\pgfcircle[fill]{\pgfxy(110.10,-10.00)}{0.50mm}
\pgfcircle[stroke]{\pgfxy(110.10,-10.00)}{0.50mm}
\pgfcircle[fill]{\pgfxy(80.10,-20.00)}{0.50mm}
\pgfcircle[stroke]{\pgfxy(80.10,-20.00)}{0.50mm}
\pgfcircle[fill]{\pgfxy(90.10,-20.00)}{0.50mm}
\pgfcircle[stroke]{\pgfxy(90.10,-20.00)}{0.50mm}
\pgfmoveto{\pgfxy(90.10,0.00)}\pgflineto{\pgfxy(80.10,-10.00)}\pgfstroke
\pgfmoveto{\pgfxy(90.10,0.00)}\pgflineto{\pgfxy(90.10,-10.00)}\pgfstroke
\pgfmoveto{\pgfxy(90.10,0.00)}\pgflineto{\pgfxy(100.10,-10.00)}\pgfstroke
\pgfmoveto{\pgfxy(90.10,0.00)}\pgflineto{\pgfxy(110.10,-10.00)}\pgfstroke
\pgfmoveto{\pgfxy(90.10,-10.00)}\pgflineto{\pgfxy(80.10,-20.00)}\pgfstroke
\pgfmoveto{\pgfxy(90.10,-10.00)}\pgflineto{\pgfxy(90.10,-20.00)}\pgfstroke
\pgfputat{\pgfxy(86.10,-11.00)}{\pgfbox[bottom,left]{\fontsize{11.38}{13.66}\selectfont 6}}
\pgfputat{\pgfxy(76.10,-21.00)}{\pgfbox[bottom,left]{\fontsize{11.38}{13.66}\selectfont 3}}
\pgfputat{\pgfxy(86.10,-21.00)}{\pgfbox[bottom,left]{\fontsize{11.38}{13.66}\selectfont 7}}
\pgfputat{\pgfxy(94.10,-11.00)}{\pgfbox[bottom,left]{\fontsize{11.38}{13.66}\selectfont 13}}
\pgfputat{\pgfxy(104.10,-11.00)}{\pgfbox[bottom,left]{\fontsize{11.38}{13.66}\selectfont 14}}
\pgfputat{\pgfxy(136.10,-1.00)}{\pgfbox[bottom,left]{\fontsize{11.38}{13.66}\selectfont 4}}
\pgfputat{\pgfxy(126.10,-11.00)}{\pgfbox[bottom,left]{\fontsize{11.38}{13.66}\selectfont 1}}
\pgfcircle[fill]{\pgfxy(140.10,0.00)}{0.50mm}
\pgfcircle[stroke]{\pgfxy(140.10,0.00)}{0.50mm}
\pgfcircle[fill]{\pgfxy(130.10,-10.00)}{0.50mm}
\pgfcircle[stroke]{\pgfxy(130.10,-10.00)}{0.50mm}
\pgfcircle[fill]{\pgfxy(140.10,-10.00)}{0.50mm}
\pgfcircle[stroke]{\pgfxy(140.10,-10.00)}{0.50mm}
\pgfcircle[fill]{\pgfxy(150.10,-10.00)}{0.50mm}
\pgfcircle[stroke]{\pgfxy(150.10,-10.00)}{0.50mm}
\pgfcircle[fill]{\pgfxy(160.10,-10.00)}{0.50mm}
\pgfcircle[stroke]{\pgfxy(160.10,-10.00)}{0.50mm}
\pgfcircle[fill]{\pgfxy(130.10,-20.00)}{0.50mm}
\pgfcircle[stroke]{\pgfxy(130.10,-20.00)}{0.50mm}
\pgfcircle[fill]{\pgfxy(140.10,-20.00)}{0.50mm}
\pgfcircle[stroke]{\pgfxy(140.10,-20.00)}{0.50mm}
\pgfcircle[fill]{\pgfxy(160.10,-20.00)}{0.50mm}
\pgfcircle[stroke]{\pgfxy(160.10,-20.00)}{0.50mm}
\pgfmoveto{\pgfxy(140.10,0.00)}\pgflineto{\pgfxy(130.10,-10.00)}\pgfstroke
\pgfmoveto{\pgfxy(140.10,0.00)}\pgflineto{\pgfxy(140.10,-10.00)}\pgfstroke
\pgfmoveto{\pgfxy(140.10,0.00)}\pgflineto{\pgfxy(150.10,-10.00)}\pgfstroke
\pgfmoveto{\pgfxy(140.10,0.00)}\pgflineto{\pgfxy(160.10,-10.00)}\pgfstroke
\pgfmoveto{\pgfxy(160.10,-10.00)}\pgflineto{\pgfxy(160.10,-20.00)}\pgfstroke
\pgfmoveto{\pgfxy(140.10,-10.00)}\pgflineto{\pgfxy(130.10,-20.00)}\pgfstroke
\pgfmoveto{\pgfxy(140.10,-10.00)}\pgflineto{\pgfxy(140.10,-20.00)}\pgfstroke
\pgfputat{\pgfxy(136.10,-11.00)}{\pgfbox[bottom,left]{\fontsize{11.38}{13.66}\selectfont 6}}
\pgfputat{\pgfxy(156.10,-21.00)}{\pgfbox[bottom,left]{\fontsize{11.38}{13.66}\selectfont 8}}
\pgfputat{\pgfxy(126.10,-21.00)}{\pgfbox[bottom,left]{\fontsize{11.38}{13.66}\selectfont 3}}
\pgfputat{\pgfxy(136.10,-21.00)}{\pgfbox[bottom,left]{\fontsize{11.38}{13.66}\selectfont 7}}
\pgfputat{\pgfxy(144.10,-11.00)}{\pgfbox[bottom,left]{\fontsize{11.38}{13.66}\selectfont 13}}
\pgfputat{\pgfxy(154.10,-11.00)}{\pgfbox[bottom,left]{\fontsize{11.38}{13.66}\selectfont 14}}
\pgfsetdash{{0.60mm}{0.50mm}}{0mm}\pgfellipse[stroke]{\pgfxy(33.50,20.00)}{\pgfxy(9.50,0.00)}{\pgfxy(0.00,3.00)}
\pgfsetdash{}{0mm}\pgfellipse[stroke]{\pgfxy(38.50,10.00)}{\pgfxy(14.50,0.00)}{\pgfxy(0.00,3.00)}
\pgfsetdash{{0.60mm}{0.50mm}}{0mm}\pgfellipse[stroke]{\pgfxy(63.50,10.00)}{\pgfxy(9.50,0.00)}{\pgfxy(0.00,3.00)}
\pgfellipse[stroke]{\pgfxy(83.50,20.00)}{\pgfxy(9.50,0.00)}{\pgfxy(0.00,3.00)}
\pgfellipse[stroke]{\pgfxy(98.50,20.00)}{\pgfxy(4.50,0.00)}{\pgfxy(0.00,3.00)}
\pgfsetdash{}{0mm}\pgfellipse[stroke]{\pgfxy(113.50,10.00)}{\pgfxy(9.50,0.00)}{\pgfxy(0.00,3.00)}
\pgfsetdash{{0.60mm}{0.50mm}}{0mm}\pgfellipse[stroke]{\pgfxy(133.50,20.00)}{\pgfxy(9.50,0.00)}{\pgfxy(0.00,3.00)}
\pgfsetdash{}{0mm}\pgfellipse[stroke]{\pgfxy(148.50,20.00)}{\pgfxy(4.50,0.00)}{\pgfxy(0.00,3.00)}
\pgfsetdash{{0.60mm}{0.50mm}}{0mm}\pgfellipse[stroke]{\pgfxy(158.50,20.00)}{\pgfxy(4.50,0.00)}{\pgfxy(0.00,3.00)}
\pgfsetdash{}{0mm}\pgfellipse[stroke]{\pgfxy(158.50,-20.00)}{\pgfxy(4.50,0.00)}{\pgfxy(0.00,3.00)}
\pgfsetdash{{0.60mm}{0.50mm}}{0mm}\pgfellipse[stroke]{\pgfxy(133.50,-20.00)}{\pgfxy(9.50,0.00)}{\pgfxy(0.00,3.00)}
\pgfsetdash{}{0mm}\pgfellipse[stroke]{\pgfxy(83.50,-20.00)}{\pgfxy(9.50,0.00)}{\pgfxy(0.00,3.00)}
\pgfellipse[stroke]{\pgfxy(43.50,-10.00)}{\pgfxy(19.50,0.00)}{\pgfxy(0.00,3.00)}
\end{pgfpicture}%
%\caption{An example for $r=6$ and $\lambda= 0^7 1^6 2^1 3^1 4^1$.}
\caption{An example of Pr\"ufer-like algorithm}
\label{prufertree}
\end{figure}

For a fixed $r \in [n]$,
let $\Tnkr$ be the set of trees on $[n]$ rooted at $r$ with $k$ non-empty child-groups.
We first define two preliminary mappings:

%%%%%%%%%%%%%%%%%%%%%%%%%%%%%%%%%%%%%%%%%%%
\subtitle{The sibship mapping $\phiglo: \Tnkr\to \Pinkr$}
For each $T \in \Tnkr$, let $\phiglo(T)$ be the set of all child-groups of $T$.
\medskip

Clearly, we have $\typeglo(T) = \type(\phiglo(T))$, and if $\lambda = \typeglo(T)$, then
$k=\ell(\lambda)$.
\subtitle{The paternity mapping $\psi: \Tnkr\to \Snkr$}
Starting from $T_{0}=T \in \Tnkr$, for $i=1,\ldots, k$, let $T_i$ be the tree obtained from $T_{i-1}$ by deleting the largest leaf-group $L_i$,
 set $\psi(T)=(p_1,p_2,\ldots,p_k)$, where $p_i$ is the parent of child-group $L_i$ in the tree $T_{i-1}$.

For example, the tree $T_0$ in Figure~\ref{prufertree} is rooted at  $r=4$ and
the non-empty child-groups of $T_0$ are:
\begin{align*}
G_4=\set{1,6,13,14}, \; G_6=\set{3,7}, \;G_8=\set{2,11}, \; G_{10}=\set{5,9,12}, \; G_{13}=\set{10}, \; G_{14}=\set{8},
\end{align*}
of which only $G_6$, $G_8$, and $G_{10}$ are the leaf-groups.
Hence
$$
\phiglo(T_0)=\set{G_{4}, G_{6}, G_{8}, G_{10}, G_{13}, G_{14}},
$$
and the maximal leaf-groups   in the  trees $T_0, \ldots, T_5$  are, respectively,
$$
L_1=G_{10},\quad L_2= G_{8},\quad
L_3= G_{13},\quad  L_4=G_{14},\quad  L_5= G_{6}, \quad L_6=G_{4}.
$$
So $\psi(T_0)=(10,8,13,14,6,4)$.

By construction, we have $\phiglo(T_i) = \phiglo(T_{i-1}) \setminus \set{L_i}$ for all $i \ge 0$,
so $L_i$ belongs to $\phiglo(T)$ for all $i$.
Since the number of child-groups of $T\in\Tnkr$ is equal to $k=\ell(\lambda)$,
this implies that $p_k = r$.
Because each child-group is deleted only once, the corresponding non-leaf vertex (parent) appears in $\psi(T)$ once and only once.
This means that $(p_{1},\ldots, p_k)$ is a $k$-permutation in $\Snkr$.
The following result shows that the pair of mappings $(\phiglo, \psi)$ defines a {\em Pr\"ufer-like algorithm} for rooted labeled trees.
\begin{thm}\label{prufer}
For all $k \in [n-1]$, the mapping $T \mapsto (\phiglo(T), \psi(T))$
is a bijection from $\Tnkr$ to $\Pinkr\times \Snkr$
%where $T \in \Tlr$ and $(\pi, \mathbf{p})\in \Pinkr\times \Snkr$
such that $$\typeglo(T) = \type(\phiglo(T)).$$
\end{thm}
\begin{proof}
Given a partition $\pi=\set{\pi_1,\ldots,\pi_k} \in \Pinkr$ and
a $k$-permutation $\mathbf{p}=(p_1,\ldots,p_{k}) \in \Snkr$,
we can construct the tree $T$ in $\Tnkr$ as follows.
For $i=1,2,\ldots,k$:
\begin{enumerate}[(a)]
\item \label{property} Order the blocks according to their maximal elements as in \eqref{eq:order}.
Let $L_i$ be the largest block of $\pi \setminus \set{L_1, \ldots, L_{i-1}}$, which does not contain any number in $\set{p_i, p_{i+1}, \ldots, p_{k-1}}$.
\item Join each vertex in $L_i$ and $p_i$ by an edge.
\end{enumerate}
The existence of the block $L_i$ in \eqref{property} can be justified by a counting argument: there remain $k-(i-1)$ blocks in $\pi \setminus \set{L_1, \ldots, L_{i-1}}$ and we have to avoid $k-i$ values in $\set{p_i, p_{i+1}, \ldots, p_{k-1}}$, so there is at least one block without any of those values.
\end{proof}

For  example, if ${\bf p}=(10,8,13,14,6,4)\in \mathcal{S}_{14,6}^{(4)}$ and
$$\pi=\set{\set{1,6,13,14}, \set{5,9,12}, \set{2,11}, \set{10}, \set{8}, \set{3,7}}\in \Pi_{14,6}^{(4)},
$$
then
the inverse Pr\"ufer-like algorithm yields $L_{1}, \ldots, L_{6}$ as follows:
\begin{alignat*}{3}
L_1&=\set{5,9,2}, \quad & L_2&=\set{2,11}, \quad & L_3&=\set{10},\\
L_4&=\set{8}, \quad & L_5&=\set{3,7}, \quad & L_6&=\set{1,6,13,14}.
\end{alignat*}
Joining  each vertex in $L_i$  with $p_i$ ($1\leq i\leq 6$)  by an edge we recover  the tree $T_0$ in Figure~\ref{prufertree}.

%%%%%%%

%\rmk Let $\lambda$ be a partition of $n-1$ such that $\l(\lambda)=k$.
%It follows from  Theorem \ref{prufer} that the number of trees in $\Tnr$ with the global indegree sequence $\lambda$ is equal to
% $\abs{\Pi_{n,\lambda}}\cdot  \abs{\Snkr}$,
%which yields a bijective proof of  Observation~2 in view of \eqref{numbpart} and  \eqref{numbtree}.

%\section{A transformation on trees}\label{mainsection}
\section{Proof of Theorem \ref{thm:main}}\label{mainsection}

\begin{figure}[t]
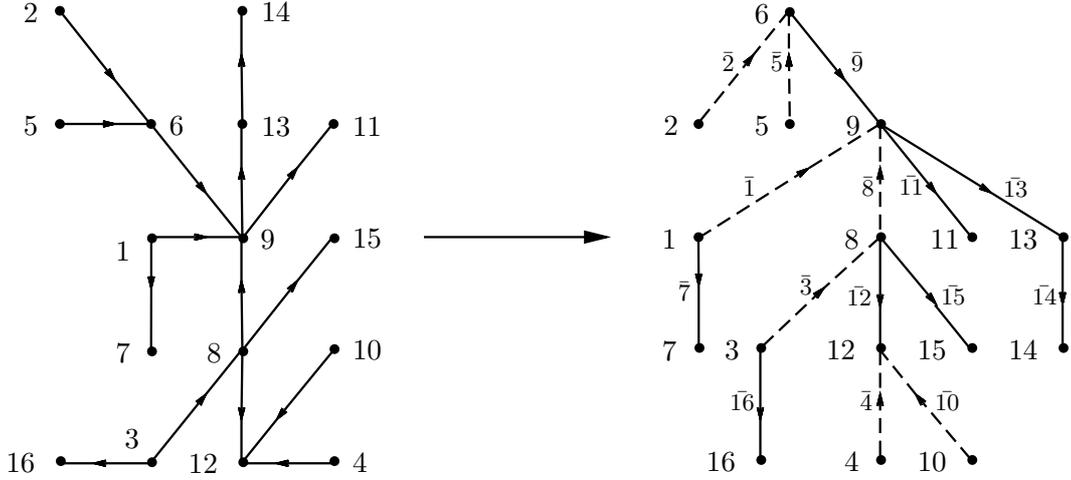

\centering
%
\caption{A tree $T$ hung up at $6$}
\label{label}
\end{figure}

Given a  tree $T\in\Tn$ and a fixed integer $r \in [n]$,
we can turn it
 as a tree rooted at $r$ by {\em hanging up} it  at $r$ as follows:
\begin{itemize}
\item
Draw the tree with  the  vertex $r$ at the top and  join $r$ to  the  vertices incident  to $r$, arranged in  increasing order from left to right,   by edges.
\item Suppose that we have drawn all the vertices with distance $i$ to $r$ (counted as the number of edges on the path to $r$), then
join each  vertex with distance $i$  to  its  incident vertices with distance $i+1$ to $r$, arranged in  increasing order from left to right;
\item Repeat the process until drawing all vertices.
 \end{itemize}
%At this time, from left to right, we arrange from the smallest child by order.
%We label the edges of $T$ on $[n]$ such that each edge has the same label as vertex right below it (See Figure~\ref{label}).
The hang-up action induces a global orientation of edges of $T$ toward the root $r$.
For  a tree  $T$ rooted at vertex $r$  we partition the edges in the following manner.
An edge is \emph{good}, respectively \emph{bad}, if its local orientation is oriented toward, respectively away from, the root $r$.
We label each edge $(vu)$ by $v$ if its global orientation is $v \to u$.
So the set of labels of all edges equals  $[n]\setminus\set{r}$ and putting  together  the labels of edges
oriented locally  toward to the same vertex yields a partition of $[n]\setminus\set{r}$, denoted by $\philoc(T)$.

%As we move downward along a decreasing (resp. increasing) edge, labels of vertices decrease (resp. increase).

For example, in Figure~\ref{label},  a tree is hung  up  at $6$,  where the dashed edges are good and
the labels of edges are barred  to avoid confusion.  The corresponding edge-label partition is
$$\phi_{loc}^{(6)}(T)=1~8~9/4~10~12/2~5/3/7/11/13/14/15/16,
$$
where the blocks are separated by a slash $/$.

Now we describe a map $\Phi_r$ from $\Tn$ to $\Tnr$, which will be shown to be a bijection.
%%%%%%%%%%%%%%%%%%
\subsection{Construction of the mapping $\Phi_r$}
We define the mapping $\Phi_r$ in three steps.
\subtitle{Step 1: Move out good  edges}
Starting from a tree $T \in \Tn$, moving out the good edges in $T$, we get a set of rooted subtrees without any  good  edges, call them {\em increasing trees},
%increasing trees (i.e., trees without any decreasing edges)
$I_T = \set{I_1,I_2, \ldots, I_d}$ and a matrix recording the cut good edges
$$
D_T =
\left(
\begin{array}{cccc}
    j_1 & j_2 & \cdots & j_{d-1} \\
    i_1 & i_2 & \cdots & i_{d-1} \\
  \end{array}
\right),
$$
where each column ${j \choose i}$ corresponds to a good edge $i \to j$ in $T$.

\rmk The  roots of the $d$ increasing trees are
$i_1,\ldots, i_{d-1}$ and $r$.

\begin{figure}[t]
\centering
%
%
%\caption{The bijection $\Phi_6 : \mathcal {T}_{loc}^{(6,\lambda)}  \to \mathcal {T}_{glo}^{(6,\lambda)}$ with $\lambda=0^6 1^7 2^1 3^2$}
\caption{The bijection $\Phi_6 : \mathcal {T}_{16} \to \mathcal {T}_{16}^{(6)}$ with $\typeloc(T)=\typeglo(T')=1^7 2^1 3^2$}
\label{diagram}
\end{figure}
For example, after cutting the good  edges, drawn with dashed arrows, in the tree $T$ of Figure~\ref{diagram}, we get
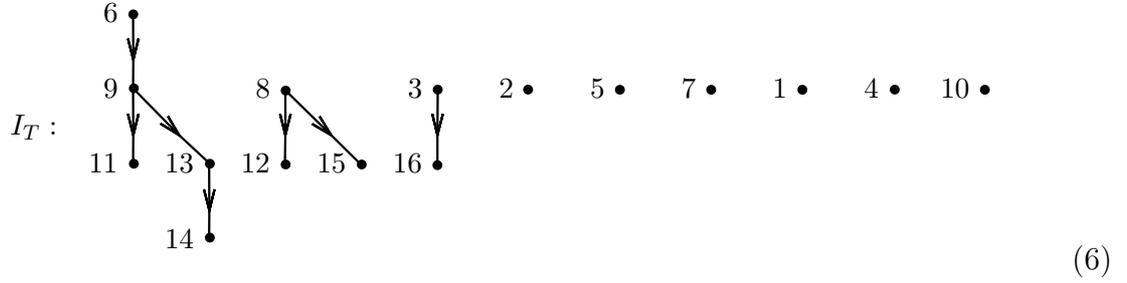
\begin{equation}
\centering
%
%
%\caption{Three increasing trees}
\label{increasing}
\end{equation}
and the matrix recording  the eight good  edges
\begin{equation}
\label{dt}
D_{T} =
\left(
  \begin{array}{cccccccc}
    6 & 6 & 7 &8& 9 & 9 & 12 & 12 \\
    2 & 5 & 1 & 3&7 & 8  & 4 & 10 \\
  \end{array}
\right).
\end{equation}

To prepare the second step, we  recall a classical linear ordering on the vertices of a tree $T$, called {\em postorder}, and denoted $\ord(T)$ (see \cite[P. 336]{Knu73}). It is defined recursively as follows:
Let $v$ be the root of $T$ and there are subtrees $T_1, \ldots, T_k$ connected to $v$. Order the subtrees $T_1, \ldots, T_k$ by their roots, then set
% defining $T$ as plane trees)
$$\ord(T) = \ord(T_1), \ldots, \ord(T_k), v \quad \text{(concatenation of words)}.$$

\begin{figure}[t]
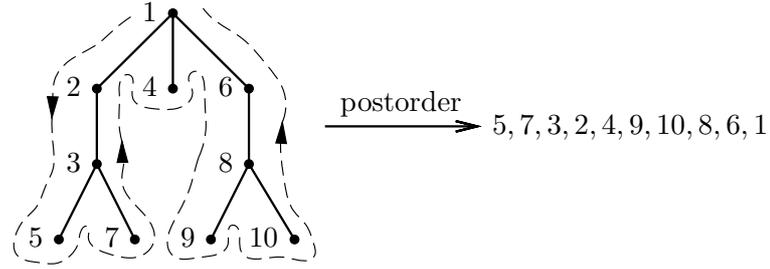

\centering
%
%
\caption{A increasing tree traversed in postorder}
\label{postorder}
\end{figure}

An example of postorder is given in Figure~\ref{postorder}.
\subtitle{Step 2: Read vertices in increasing trees in postorder}
For each increasing tree $I_h$ we construct a {\em linear tree} $J_h=v_1 \to \cdots \to v_l$, of which every vertex has at most one child, and a cyclic permutation $\sigma_h = (v_1, \ldots , v_l)$, where $v_1, \ldots, v_l$ are the vertices of $I_h$  ordered  by postorder.
So the last $v_l$ is the root of the tree $I_h$ and also the minimum in the sequence $v_1, \ldots, v_l$.
%according to an order of labels of vertices,
%Thus each tree $I_h$ gives a cyclic permutation $\sigma_h$.
%Next, to each cyclic permutation $\sigma_h=(v_1, \ldots, v_l)$ with $v_l$ as minimum
%we associate the {\em linear tree} $J_h$ as $v_1 \to \cdots \to v_l$ with $v_l$ as root.
Define $J_T=\set{J_1,\ldots, J_d}$
and the matrix
$$\sigma(D_T)=
\left(
  \begin{array}{cccc}
    \sigma(j_1) & \sigma(j_2) & \cdots & \sigma(j_{d-1}) \\
    i_1 & i_2 & \cdots & i_{d-1} \\
  \end{array}
\right),
$$
where $\sigma=\sigma_1 \ldots \sigma_d$.

In  the above example,
we have
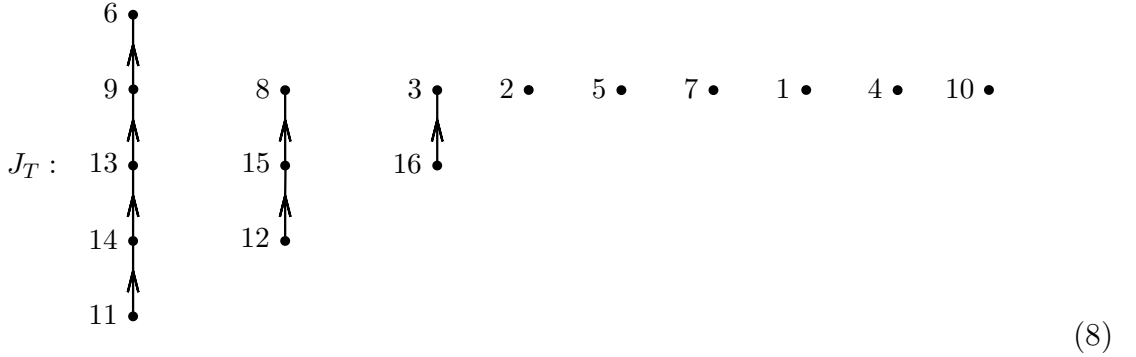
\begin{equation}
\centering
\begin{pgfpicture}{-79.11mm}{-23.86mm}{55.06mm}{24.14mm}
\pgfsetxvec{\pgfpoint{1.00mm}{0mm}}
\pgfsetyvec{\pgfpoint{0mm}{1.00mm}}
\color[rgb]{0,0,0}\pgfsetlinewidth{0.30mm}\pgfsetdash{}{0mm}
\pgfmoveto{\pgfxy(-60.00,16.00)}\pgflineto{\pgfxy(-60.00,10.00)}\pgfstroke
\pgfmoveto{\pgfxy(-60.00,16.00)}\pgflineto{\pgfxy(-60.70,13.20)}\pgflineto{\pgfxy(-60.00,16.00)}\pgflineto{\pgfxy(-59.30,13.20)}\pgflineto{\pgfxy(-60.00,16.00)}\pgfclosepath\pgffill
\pgfmoveto{\pgfxy(-60.00,16.00)}\pgflineto{\pgfxy(-60.70,13.20)}\pgflineto{\pgfxy(-60.00,16.00)}\pgflineto{\pgfxy(-59.30,13.20)}\pgflineto{\pgfxy(-60.00,16.00)}\pgfclosepath\pgfstroke
\pgfmoveto{\pgfxy(-60.00,20.00)}\pgflineto{\pgfxy(-60.00,16.00)}\pgfstroke
\pgfcircle[fill]{\pgfxy(-60.00,20.00)}{0.50mm}
\pgfcircle[stroke]{\pgfxy(-60.00,20.00)}{0.50mm}
\pgfcircle[fill]{\pgfxy(-59.96,10.11)}{0.50mm}
\pgfcircle[stroke]{\pgfxy(-59.96,10.11)}{0.50mm}
\pgfcircle[fill]{\pgfxy(-40.00,10.00)}{0.50mm}
\pgfcircle[stroke]{\pgfxy(-40.00,10.00)}{0.50mm}
\pgfcircle[fill]{\pgfxy(-60.00,0.00)}{0.50mm}
\pgfcircle[stroke]{\pgfxy(-60.00,0.00)}{0.50mm}
\pgfcircle[fill]{\pgfxy(-60.00,-10.00)}{0.50mm}
\pgfcircle[stroke]{\pgfxy(-60.00,-10.00)}{0.50mm}
\pgfcircle[fill]{\pgfxy(-60.00,-20.00)}{0.50mm}
\pgfcircle[stroke]{\pgfxy(-60.00,-20.00)}{0.50mm}
\pgfcircle[fill]{\pgfxy(-40.00,0.00)}{0.50mm}
\pgfcircle[stroke]{\pgfxy(-40.00,0.00)}{0.50mm}
\pgfcircle[fill]{\pgfxy(-20.00,0.00)}{0.50mm}
\pgfcircle[stroke]{\pgfxy(-20.00,0.00)}{0.50mm}
\pgfputat{\pgfxy(-62.00,19.00)}{\pgfbox[bottom,left]{\fontsize{11.38}{13.66}\selectfont \makebox[0pt][r]{6}}}
\pgfputat{\pgfxy(-42.00,9.00)}{\pgfbox[bottom,left]{\fontsize{11.38}{13.66}\selectfont \makebox[0pt][r]{8}}}
\pgfputat{\pgfxy(-62.00,9.00)}{\pgfbox[bottom,left]{\fontsize{11.38}{13.66}\selectfont \makebox[0pt][r]{9}}}
\pgfputat{\pgfxy(-61.96,-0.88)}{\pgfbox[bottom,left]{\fontsize{11.38}{13.66}\selectfont \makebox[0pt][r]{13}}}
\pgfputat{\pgfxy(-41.90,-0.90)}{\pgfbox[bottom,left]{\fontsize{11.38}{13.66}\selectfont \makebox[0pt][r]{15}}}
\pgfputat{\pgfxy(-62.00,-11.00)}{\pgfbox[bottom,left]{\fontsize{11.38}{13.66}\selectfont \makebox[0pt][r]{14}}}
\pgfputat{\pgfxy(-62.00,-21.00)}{\pgfbox[bottom,left]{\fontsize{11.38}{13.66}\selectfont \makebox[0pt][r]{11}}}
\pgfputat{\pgfxy(-22.00,-1.00)}{\pgfbox[bottom,left]{\fontsize{11.38}{13.66}\selectfont \makebox[0pt][r]{16}}}
\pgfmoveto{\pgfxy(-60.00,6.00)}\pgflineto{\pgfxy(-60.00,0.00)}\pgfstroke
\pgfmoveto{\pgfxy(-60.00,6.00)}\pgflineto{\pgfxy(-60.70,3.20)}\pgflineto{\pgfxy(-60.00,6.00)}\pgflineto{\pgfxy(-59.30,3.20)}\pgflineto{\pgfxy(-60.00,6.00)}\pgfclosepath\pgffill
\pgfmoveto{\pgfxy(-60.00,6.00)}\pgflineto{\pgfxy(-60.70,3.20)}\pgflineto{\pgfxy(-60.00,6.00)}\pgflineto{\pgfxy(-59.30,3.20)}\pgflineto{\pgfxy(-60.00,6.00)}\pgfclosepath\pgfstroke
\pgfmoveto{\pgfxy(-60.00,10.00)}\pgflineto{\pgfxy(-60.00,6.00)}\pgfstroke
\pgfmoveto{\pgfxy(-40.00,6.00)}\pgflineto{\pgfxy(-40.00,0.00)}\pgfstroke
\pgfmoveto{\pgfxy(-40.00,6.00)}\pgflineto{\pgfxy(-40.70,3.20)}\pgflineto{\pgfxy(-40.00,6.00)}\pgflineto{\pgfxy(-39.30,3.20)}\pgflineto{\pgfxy(-40.00,6.00)}\pgfclosepath\pgffill
\pgfmoveto{\pgfxy(-40.00,6.00)}\pgflineto{\pgfxy(-40.70,3.20)}\pgflineto{\pgfxy(-40.00,6.00)}\pgflineto{\pgfxy(-39.30,3.20)}\pgflineto{\pgfxy(-40.00,6.00)}\pgfclosepath\pgfstroke
\pgfmoveto{\pgfxy(-40.00,10.00)}\pgflineto{\pgfxy(-40.00,6.00)}\pgfstroke
\pgfmoveto{\pgfxy(-60.00,-14.00)}\pgflineto{\pgfxy(-60.00,-20.00)}\pgfstroke
\pgfmoveto{\pgfxy(-60.00,-14.00)}\pgflineto{\pgfxy(-60.70,-16.80)}\pgflineto{\pgfxy(-60.00,-14.00)}\pgflineto{\pgfxy(-59.30,-16.80)}\pgflineto{\pgfxy(-60.00,-14.00)}\pgfclosepath\pgffill
\pgfmoveto{\pgfxy(-60.00,-14.00)}\pgflineto{\pgfxy(-60.70,-16.80)}\pgflineto{\pgfxy(-60.00,-14.00)}\pgflineto{\pgfxy(-59.30,-16.80)}\pgflineto{\pgfxy(-60.00,-14.00)}\pgfclosepath\pgfstroke
\pgfmoveto{\pgfxy(-60.00,-10.00)}\pgflineto{\pgfxy(-60.00,-14.00)}\pgfstroke
\pgfmoveto{\pgfxy(-20.00,6.00)}\pgflineto{\pgfxy(-20.00,0.00)}\pgfstroke
\pgfmoveto{\pgfxy(-20.00,6.00)}\pgflineto{\pgfxy(-20.70,3.20)}\pgflineto{\pgfxy(-20.00,6.00)}\pgflineto{\pgfxy(-19.30,3.20)}\pgflineto{\pgfxy(-20.00,6.00)}\pgfclosepath\pgffill
\pgfmoveto{\pgfxy(-20.00,6.00)}\pgflineto{\pgfxy(-20.70,3.20)}\pgflineto{\pgfxy(-20.00,6.00)}\pgflineto{\pgfxy(-19.30,3.20)}\pgflineto{\pgfxy(-20.00,6.00)}\pgfclosepath\pgfstroke
\pgfmoveto{\pgfxy(-20.00,10.00)}\pgflineto{\pgfxy(-20.00,6.00)}\pgfstroke
\pgfcircle[fill]{\pgfxy(-20.00,10.00)}{0.50mm}
\pgfcircle[stroke]{\pgfxy(-20.00,10.00)}{0.50mm}
\pgfputat{\pgfxy(-22.00,9.00)}{\pgfbox[bottom,left]{\fontsize{11.38}{13.66}\selectfont \makebox[0pt][r]{3}}}
\pgfcircle[fill]{\pgfxy(-7.97,10.02)}{0.50mm}
\pgfcircle[stroke]{\pgfxy(-7.97,10.02)}{0.50mm}
\pgfputat{\pgfxy(-9.89,9.00)}{\pgfbox[bottom,left]{\fontsize{11.38}{13.66}\selectfont \makebox[0pt][r]{2}}}
\pgfcircle[fill]{\pgfxy(4.22,10.00)}{0.50mm}
\pgfcircle[stroke]{\pgfxy(4.22,10.00)}{0.50mm}
\pgfputat{\pgfxy(2.22,9.00)}{\pgfbox[bottom,left]{\fontsize{11.38}{13.66}\selectfont \makebox[0pt][r]{5}}}
\pgfcircle[fill]{\pgfxy(16.33,10.00)}{0.50mm}
\pgfcircle[stroke]{\pgfxy(16.33,10.00)}{0.50mm}
\pgfputat{\pgfxy(14.33,9.00)}{\pgfbox[bottom,left]{\fontsize{11.38}{13.66}\selectfont \makebox[0pt][r]{7}}}
\pgfcircle[fill]{\pgfxy(28.53,10.02)}{0.50mm}
\pgfcircle[stroke]{\pgfxy(28.53,10.02)}{0.50mm}
\pgfputat{\pgfxy(26.44,9.00)}{\pgfbox[bottom,left]{\fontsize{11.38}{13.66}\selectfont \makebox[0pt][r]{1}}}
\pgfcircle[fill]{\pgfxy(40.55,10.00)}{0.50mm}
\pgfcircle[stroke]{\pgfxy(40.55,10.00)}{0.50mm}
\pgfputat{\pgfxy(38.55,9.00)}{\pgfbox[bottom,left]{\fontsize{11.38}{13.66}\selectfont \makebox[0pt][r]{4}}}
\pgfcircle[fill]{\pgfxy(52.56,10.02)}{0.50mm}
\pgfcircle[stroke]{\pgfxy(52.56,10.02)}{0.50mm}
\pgfputat{\pgfxy(50.66,9.00)}{\pgfbox[bottom,left]{\fontsize{11.38}{13.66}\selectfont \makebox[0pt][r]{10}}}
\pgfmoveto{\pgfxy(-60.00,-4.00)}\pgflineto{\pgfxy(-60.00,-10.00)}\pgfstroke
\pgfmoveto{\pgfxy(-60.00,-4.00)}\pgflineto{\pgfxy(-60.70,-6.80)}\pgflineto{\pgfxy(-60.00,-4.00)}\pgflineto{\pgfxy(-59.30,-6.80)}\pgflineto{\pgfxy(-60.00,-4.00)}\pgfclosepath\pgffill
\pgfmoveto{\pgfxy(-60.00,-4.00)}\pgflineto{\pgfxy(-60.70,-6.80)}\pgflineto{\pgfxy(-60.00,-4.00)}\pgflineto{\pgfxy(-59.30,-6.80)}\pgflineto{\pgfxy(-60.00,-4.00)}\pgfclosepath\pgfstroke
\pgfmoveto{\pgfxy(-60.00,0.00)}\pgflineto{\pgfxy(-60.00,-4.00)}\pgfstroke
\pgfcircle[fill]{\pgfxy(-40.00,-10.00)}{0.50mm}
\pgfcircle[stroke]{\pgfxy(-40.00,-10.00)}{0.50mm}
\pgfmoveto{\pgfxy(-40.00,-4.00)}\pgflineto{\pgfxy(-40.00,-10.00)}\pgfstroke
\pgfmoveto{\pgfxy(-40.00,-4.00)}\pgflineto{\pgfxy(-40.70,-6.80)}\pgflineto{\pgfxy(-40.00,-4.00)}\pgflineto{\pgfxy(-39.30,-6.80)}\pgflineto{\pgfxy(-40.00,-4.00)}\pgfclosepath\pgffill
\pgfmoveto{\pgfxy(-40.00,-4.00)}\pgflineto{\pgfxy(-40.70,-6.80)}\pgflineto{\pgfxy(-40.00,-4.00)}\pgflineto{\pgfxy(-39.30,-6.80)}\pgflineto{\pgfxy(-40.00,-4.00)}\pgfclosepath\pgfstroke
\pgfmoveto{\pgfxy(-39.95,0.08)}\pgflineto{\pgfxy(-40.00,-4.00)}\pgfstroke
\pgfputat{\pgfxy(-42.01,-10.77)}{\pgfbox[bottom,left]{\fontsize{11.38}{13.66}\selectfont \makebox[0pt][r]{12}}}
\pgfputat{\pgfxy(-70.00,-1.00)}{\pgfbox[bottom,left]{\fontsize{11.38}{13.66}\selectfont \makebox[0pt][r]{$J_T:$}}}
\end{pgfpicture}%
%\caption{Three increasing trees}
\label{lt}
\end{equation}
and three non-identical cyclic permutations corresponding to the first three trees:
%(the singleton trees give the identical permutation)
\begin{equation}
\label{per}
\sigma_1=(11,14,13,9,6), \quad \sigma_2=(12,15,8),\quad \text{and}\quad\sigma_3=(16,3).
\end{equation}
Applying $\sigma$ to the matrix \eqref{dt}, we obtain
the matrix
\begin{equation}
\label{sdt}
\sigma(D_{T}) =
\left(
  \begin{array}{cccccccc}
    11 & 11 & 7 &12& 6 & 6  & 15 & 15 \\
    2 & 5 & 1 &3& 7 & 8  & 4 & 10 \\
  \end{array}
\right).
\end{equation}

For a graph $G$, let $V(G)$ be the set of all vertices in $G$.
Define the relation $\sim_G$ on its vertices  as follows:
$$
a \sim_G b \Leftrightarrow \text{$a,b$ are connected by a path in $G$ regardless of an orientation}.
$$
By definition, $I_T$ and $J_T$ are graphs with $d$ connected components. We shall identify an edge $i\rightarrow j$ with the column ${j\choose i}$ in the matrix $D_T$ and  $\sigma(D_T)$.

\begin{lem}
\label{connected}
In Step 2, for any vertex $v \not \sim_{J_T} r$, there is a unique sequence of edges
 ${\sigma(j_1) \choose i_1}$, ${\sigma(j_2) \choose i_2}, \ldots,
{\sigma(j_l) \choose i_l}$ in $\sigma(D_T)$  such that
\begin{align}\label{eq:chain}
v \sim_{J_T} i_1, \sigma(j_1) \sim_{J_T} i_2, \cdots, \sigma(j_{l-1})\sim_{J_T} i_l, \text{ and } \sigma(j_l) \sim_{J_T} r.
\end{align}
\end{lem}
\begin{proof}
Since two connected components including $r$ in $I_T$ and $J_T$ have the same vertices, $v \not \sim_{J_T} r$ implies $v \not \sim_{I_T} r$.
Since $T$ is a tree (so connected), for any vertex $v \not \sim_{I_T} r$,
there is a unique sequence of good edges
$i_1 \to j_1, i_2 \to j_2, \ldots, i_l \to j_l$ such that $$v \sim_{I_T} i_1, j_1 \sim_{I_T} i_2, \cdots, j_{l-1}\sim_{I_T} i_l, \text{ and } j_l \sim_{I_T} r.$$
Since $V(I_h)= V(J_h)$ for all $h$ and $j \sim_{J_T} \sigma(j)$ for all $j$,
the edges ${\sigma(j_1) \choose i_1}, {\sigma(j_2) \choose i_2}, \ldots, {\sigma(j_l) \choose i_l}$ in $\sigma(D_T)$ satisfy the condition \eqref{eq:chain}.
\end{proof}

\ex In the previous example with $r=6$, if $v=10$ then the unique sequence of edges
in \eqref{sdt} satisfying \eqref{eq:chain} is ${15\choose 10}$ and ${6\choose 8}$.

\subtitle{Step 3: Construct  the rooted tree}
By Lemma~\ref{connected},
the linear trees in $J_{T}$ are connected by edges $i\to j$, where $j \choose i$ is a column in the matrix $\sigma(D_T)$.
This yields a tree $\Phi_r(T)$ rooted at $r$ (with the global orientation).

An example of the map $\Phi_r$ with step 3 is illustrated  in Figure~\ref{diagram},
where steps 1 and 2 are given in \eqref{increasing} and \eqref{dt}, \eqref{lt} and \eqref{sdt}.

Next we have to  show that the map $\Phi_r$ is a bijection.  As suggested by a referee, it is convenient to summarize the key properties of $\Phi_r$ before  the proof.
%%%%%%%%%%%%%%%%%%%%%
\subsection{Key properties of $\Phi_r$}
We denote by $I_{T}:=(I_h)_h$ the connected components of the graph made up of the bad edges, some components may be reduced to a single vertex. Each component $I_h$ contains a (spanning) tree made up of bad edges that is rooted at the vertex $r_h$ which is at minimal distance to the root $r$ among the vertices of $I_h$. If $r_h\ne r$, the path from $r_h$ to $r$ starts with an edge $e_h$ called the \emph{rooting edge} of $I_h$.
By definition, an edge is a rooting edge if and only if it is  a good edge.
Each component $I_h$ defines an edge set $C_h$ made up of the bad edges between two vertices of $I_h$ and the good edges incident to a vertex of $I_h$, except the rooting edge $e_h$, if any. The sets $(C_h)_h$ forms a partition of the edges of $T$: bad edge's endpoints appear in a single $I_h$ and a good edge is the rooting edge of one of its endpoint  and thus appears in the component defined by its other endpoint. All edges contributing to the local indegree of a vertex $v \in I_h$ in $T$ belong to $C_h$. The bijection will be defined independently on each set $C_h$ using only the additional (global) information of the root vertex $r_h$. The possible components $C_h$ are the trees rooted at $r_h$ where any child with a label lower than the label of its parent is a leaf. For any vertex $v \in I_h$ we denote by
$$L_h(v) := \set{w:(wv) \in C_h\text{ and }w \not\in I_h} $$
the set of its lower children, since $\forall w \in L_h(v)$, $w < v$. The post-order linear ordering of the vertices of $I_h$ leads to a cyclic permutation $\sigma_h$ of the vertices of $I_h$.

 The transformation by postorder  leads to a graph where for any vertex $v \ne r_h$ in $I_h$, the vertex $v$ and $L_{h}(v)$ form the sibship of the vertex $\sigma_h(v)$, so $v$ is the member of this new sibship with the biggest label. Moreover, the local indegree of $v$ was $1 +\abs{L_h(v)}$ and the new global degree of $\sigma_h(v)$ is the same. In the case of $r_h$ of local indegree $0 + \abs{L_{h}(r_h)}$, its lower children of $L_{h}(r_h)$ become the sibship of another vertex $v_l$ of $I_h$ whose new global indegree is also $0 + \abs{L_{h}(r_h)}$. In addition, all the vertices of $L_{h}(r_h)$, if any, are smaller than $r_h$ in particular the biggest label among $L_{h}(r_h)$. Thus the distribution local indegrees of vertices of $I_h$ becomes the distribution of global indegrees of vertices of $J_h$ after the transformation.
%  Moreover, we can test if a vertex at maximal distance from $r_h$ in the new graph initially belongs to $I_h$: we have to check if the maximal label is greater than $r_h$. We will use these remarks in the construction of the inverse and the proof that the indegree distributions are preserved.
%%%%%%%%%%%%%%%%%%%%%%%%%%%%%%%%
\subsection{Construction of the inverse mapping $\Phi_r^{-1}$}
Let $T\in \Tnr$. First we need to introduce some
definitions.
If $i \to j$ is an edge of $T$, we say that the vertex $i$ is a {\em child} of $j$.
The vertex $i$ is the {\em eldest child} of $j$ if $i$ is bigger than all other children (if any) of $j$ and the edge $i \to j$ is {\em eldest} if $i$ is the eldest child of $j$.
Note that deleting all non-eldest edges in $T$, we obtain a set of linear trees.
For a linear tree $v_1 \to \cdots \to v_l$ obtained from $T$ by deleting all non-eldest edges, an edge $i \to j$ is called a {\em minimal} if $i$ is a right-to-left minimum in the sequence $v_1, \ldots, v_l$.
Finally, an edge $i\to j$ of $T$  is {\em proper} if it is non-eldest or minimal.
% \begin{enumerate}
% \item it is a non-eldest edge or
% \item it is a minimal edge.
% \end{enumerate}

For example, for the tree $T'$ in Figure~\ref{diagram}, the proper edges are dashed.
Moreover, the edges $7\to 6$, $8\to 6$, $4\to 15$, $10\to 15$ and $2\to 11$ are non-eldest, while $3\to 12$, $1\to 7$ and $5\to 11$ are minimal.

\begin{lem}\label{dash}
For a given tree $T$ with its local orientation, every improper edge $i \to j$ in $\Phi_r(T)$ corresponds to a column $j \choose i$ in $\sigma(D_T)$.
\end{lem}

\begin{proof}
Let $i \to j$ be an edge in $\Phi_r(T)$ corresponding to a column $j \choose i$ in $\sigma(D_T)$.
Let $k=\sigma^{-1}(j)$.
Since $j \choose i$ is induced from a good edge $i \to k$, we have $i < k$.
Denote by $J$ the linear tree including $j$ obtained from $T$ by steps 1 and 2.
\begin{enumerate}
\item If $j$ is a non-leaf of $J$, then $k$ is a child of $j$. So $i$ cannot be the eldest child of $j$ and the edge $i \to j$ must be proper in $\Phi_r(T)$.
\item If $j$ is a leaf of $J$, then $J=j \to \cdots \to k$. Suppose that there exists another column $j \choose i'$ in $\sigma(D_T)$ such that $i' > i$, then the vertex $i$ cannot be the eldest child of $j$ and the edge $i \to j$ should be proper in $\Phi_r(T)$. Otherwise, since $k$ is also the minimum of $J$ and $i<k$, the vertex $i$ is smaller than all vertices between $j$ and $k$. That means the edge $i \to j$ is minimal in the linear tree $i \to j \to \cdots \to k$. Thus the edge $i \to j$ should be proper in $\Phi_r(T)$.
\end{enumerate}

Conversely, let $i \to j$ be an edge in $\Phi_r(T)$ such that $j \choose i$ is not a column in $\sigma(D_T)$. Since the edge $i \to j$ is obtained from some linear tree $J$, we have  $j=\sigma(i)$.
If $j$ has another child $k$ in $\Phi_r(T)$, then $j \choose k$ is a column in $\sigma(D_T)$. Since $j \choose k$ is induced from a good  edge, $k \to i$ implies $k<i$. That means the edge $i \to j$ is always eldest in $\Phi_r(T)$. Since $i$ is also bigger than the root of $J$, the edge  $i \to j$ cannot be minimal. Thus the edge $i \to j$ is not proper.
\end{proof}

The following two lemmas are our main results of this section.
\begin{lem}\label{main}
The map $\Phi_r: T \mapsto T'$ is a bijection from $\Tn$ to $\Tnr$.
% such that the sibship of the vertex $v$ in $T$ is the same with the sibship of the vertex $\sigma(v)$ in $\Phi_r(T)$.
\end{lem}

\begin{proof}
% By Lemma~\ref{sibship},
It suffices to define the inverse procedure.
Given a tree $T' \in \Tnr$, by cutting out all the proper edges in $T'$,
we get a set of linear trees (i.e., trees without any proper edges including singleton vertex)
$J_{T'} = \set{J_1,J_2, \ldots, J_d}$
and a matrix recording the cut proper edges
$$
P_{T'} =
\left(
  \begin{array}{cccc}
    j_1 & j_2 & \cdots & j_{d-1} \\
    i_1 & i_2 & \cdots & i_{d-1} \\
  \end{array}
\right)
$$
where each column ${j \choose i}$ corresponds to a proper edge $i \to j$ in $T'$.
Lemma~\ref{dash} yields $P_{\Phi_r(T)} = \sigma(D_{T})$ for any $T \in \Tn$.
%Given a tree $T \in \Tlr(\pi)$, for each vertex $v$ of $T$, we cut out all the edges connecting $v$ to its children except that connecting to the eldest child
For example, for the tree $T'$ in Figure~\ref{diagram}, we obtain the nine linear trees in \eqref{lt}
%\begin{figure}[h]
%\input{merge.TpX}
%%\caption{The bijection $\Phi_6 : T_1 \mapsto T_2$}
%\end{figure}
and the matrix in \eqref{sdt}.
%$$
%P_{T_2} =
%\left(
%  \begin{array}{cccccccc}
%    11 & 11 & 7 & 6 & 6 & 12 & 15 & 15 \\
%    2 & 5 & 1 & 7 & 8 & 3 & 4 & 10 \\
%  \end{array}
%\right).
%$$
%we cut out the edges $7 \to 6$, $8 \to 6$, $1\to 7$, $4 \to 15$, $10 \to 15$, $3 \to 12$, $2 \to 11$, and $5 \to 11$ and get three linear trees $$6\leftarrow 9\leftarrow 13\leftarrow 14\leftarrow 11, \quad 8\leftarrow 15\leftarrow 12, \quad \text{and} \quad 3\leftarrow 16.$$
%we cut out edges before left-to-right minimums of each linear tree, then we get more linear trees, which

To each linear tree $J_h=v_1 \to \cdots \to v_l$ with $v_l$ as root we associate the cyclic permutation $\sigma_h=(v_1, \ldots, v_l)$ and let $\sigma=\sigma_1 \ldots \sigma_d$. For the tree $T'$ in Figure~\ref{diagram}, we get the three non-trivial permutations in \eqref{per}. %$\sigma_1=(11,14,13,9,6)$, $\sigma_2=(12,15,8)$, and $\sigma_3=(16,3)$.

Define the matrix
$$\sigma^{-1}(P_{T'})=
\left(
  \begin{array}{cccc}
    \sigma^{-1}(j_1) & \sigma^{-1}(j_2) & \cdots & \sigma^{-1}(j_{d-1}) \\
    i_1 & i_2 & \cdots & i_{d-1} \\
  \end{array}
\right).
$$
Since each column $j \choose i$ of $P_{T'}$ corresponds to an proper edge $i \to j$, $\sigma^{-1}(j)$ is the eldest child of $j$ or the root of the linear tree containing $j$. Thus we have $\sigma^{-1}(j) > i$ and
the columns of matrix $\sigma^{-1}(P_{T'})$ are decreasing.
Continuing above example, we recover the matrix in \eqref{dt}.

Since we read vertices of increasing trees $I_h$ in postorder in $\Phi_r$, every cyclic permutation $\sigma_h=(v_1, \ldots, v_l)$ can also be changed to increasing tree $I_h$ using the {\em inverse of postorder algorithm}, which is the well-known algorithm (see \cite[P. 25]{MR1442260}) mapping cyclic permutations to increasing trees as follows:
Given a cyclic permutation $\sigma_h=(v_1, \ldots, v_l)$ with $v_l$ as minimum,
construct an increasing tree $I_h$ on $v_1, \ldots, v_l$ with the root $v_l$ by defining vertex $v_i$ to be the child of the leftmost vertex $v_j$ in $\sigma_h$ which follows $v_i$ and which is less than $v_i$.
Since the last $v_l$ is the minimum in all vertices of $J_h$, there exists such a vertex $v_j$ for all vertex $v_i$ except of $v_l$.
For example, applying the linear trees in \eqref{lt}, we recover the increasing trees in \eqref{increasing}.

Finally, merging all increasing trees $I_h$ by the good  edges in the matrix $\sigma^{-1}(P_T)$, we recover the tree $\Phi_r^{-1}(T') \in \Tn$, as illustrated in Figure~\ref{diagram}.
\end{proof}

%%%%%%%%%%%%%%%%%%%%%
\subsection{Further properties of the mapping $\Phi_r$}
Define the {\em sibship} of a vertex $v$ in a oriented tree $T$ hung up $r$ to be the set of labels of edges pointed to $v$ in $T$ and denote it by $\sib^{(r)}(T;v)$. For instance, $\sibloc^{(6)}(T;9)=\set{\bar{1},\bar{8}, \bar{9}}$ and $\sibglo^{(6)}(T;9)=\set{\bar{1},\bar{8}, \bar{11}, \bar{13}}$ where $T$ is a tree in Figure~\ref{label}.

\begin{lem}\label{sibship}
For a given tree $T$ hung up at $r$ with the local orientation and for any vertex $v$ of $T$,
%we have $$\indeg_{T}(v) = \indeg_{T'}(\sigma(v))$$ where $T'=\Phi_r(T)$ is a rooted tree with the global orientation. Furthermore,
the sibship of the vertex $v$ in $T$ is the same as the sibship of the vertex $\sigma(v)$ in $\Phi_r(T)$, i.e.,
$$\sibloc^{(r)}(T;v) = \sibglo^{(r)}(T';\sigma(v))$$
where $T'=\Phi_r(T)$ is a rooted tree with the global orientation.
Therefore, $\philoc(T) = \phiglo(T')$.
\end{lem}
\begin{proof}
Let $T$ be a tree with the local orientation and $T'=\Phi_r(T)$. Let $\bar{k} \in \sibloc^{(r)}(T;v)$.
\begin{enumerate}
\item If $k<v$, we find a decreasing edge $k \stackrel{\bar{k}}{\to} v$. It becomes an edge $k \stackrel{\bar{k}}{\to} \sigma(v)$ in $T'$ under $\sigma$. Thus $\bar{k} \in \sibglo^{(r)}(T';\sigma(v))$.
\item If $k=v$, we find an increasing edge $i \stackrel{\bar{v}}{\to} v$ for some $i<v$. Since it is an edge in some increasing tree $I$, $v$ is not the root of $I$. Then we can find an edge $v \stackrel{\bar{v}}{\to} \sigma(v)$ in the linear tree corresponding to $I$. Thus $\bar{v} \in \sibglo^{(r)}(T';\sigma(v))$.
\item If $k>v$, the edge $k \leftarrow v$ points to $k$ which is impossible
\end{enumerate}
Since any two sibships are disjoint in $T'$, we have $$\sibloc^{(r)}(T;v) = \sibglo^{(r)}(T';\sigma(v))$$ where $T'=\Phi_r(T)$.
\end{proof}

Combining the above two lemmas we obtain Theorem~\ref{thm:main}.
%
%\begin{proof}[Proof of Theorem~\ref{thm:first}]  By Theorems~\ref{main} and \ref{sibship},
%the map $\Phi_r: T \mapsto T'$ is a bijection from $\Tn$ to $\Tnr$ such that
%\begin{align}\label{typesame}
%\typeloc(T) = \typeglo(T').
%\end{align}
%Thus, for any partition $\lambda$ of $n-1$, the number of
%labeled trees on $[n]$ with local indegree sequence $\lambda$ is equal to the number of labeled trees on $[n]$ rooted at $r$ with global indegree sequence $\lambda$. This
%proves Theorem~\ref{thm:first} by Fact~2.
%\end{proof}

\rmk
Let $r=1$.
Let $\pi$ be a partition of $ \{2,\ldots, n\}$ and  $\Tglopi$ (resp. $\Tlocpi$) be the set of trees with {\em sibship set-partition} $\pi$ induced by the sibship mapping $\phi_{glo}$ (resp. $\phi_{loc}$).
Combining two maps $\Phi_1$ and $\psi$ we obtain a bijective proof of Theorem 1.1 in \cite{DY10}.
Indeed, their set $T_{\pi}$ in \cite{DY10} is equal to our set $\Tlocpi$,  hence
$$
\abs{\Tlocpi}
\stackrel{\Phi_1}{=} \abs{\Tglopi}
 = \abs{(\phi_{glo})^{-1}(\pi)}
\stackrel{\psi}{=} \abs{\mathcal{S}_{n,\ell(\lambda)}^{(1)}}
=\dfrac{(n-1)!}{(n-\ell(\lambda))!}.
$$

At the end of their paper \cite{DY10}, Du and Yin  also asked for a bijection from $\Tnl$ to $\Pi_{n,\lambda}^{(1)} \times \mathcal{S}_{n,\ell(\lambda)}^{(1)}$ (in our notation).
 By Theorem~\ref{prufer}, the mapping $(\phi_{glo},\psi) \circ \Phi_1$ provides such a bijection. This is a generalization of Pr\"ufer code for labeled tress, which corresponds to the $\lambda=1^{n-1}$ case.

\section{Proof of Theorem \ref{thm:qbinomial}}\label{four}
Since $\qbin{n}{e_0,e_1,\ldots} [e_h]_q = [n]_q \qbin{n-1}{e_0,\ldots, e_h-1, \ldots}$,
the formula \eqref{eq:gen} is equivalent to
\begin{align}\label{eq:qbis}
\sum_{i\geq 0} \sum_{|\lambda|=m-1 \atop \ell(\lambda)\leq n}
&q^{(p+1)(m-i-1) + 2 n(\lambda) - 2\sum_{k=1}^i (\lambda'_k-1)}\nonumber\\
& \times   \qbin{p+i-l}{p} \qbin{n-1}{e_0,e_1,\ldots,e_h-1,\ldots}
=\qbin{n+m-2+p-l}{n-1+p}.
\end{align}

By using the formula \cite[Theorem 3.3]{MR1634067}
$$
(z;q)_N=\sum_{j=0}^N{N\brack j}_q (-1)^jz^jq^{j\choose 2}
$$
to expand $(z;q)_N$ and extracting the coefficient of $t^{k}$ in
$$
(-t;q)_{n+k-1}=(-t;q)_{k-1}(-tq^{k-1};q)_{n},
$$
we obtain the $q$-Chu-Vandermonde identity:
$$
\qbin{n+k-1}{k}=\sum_{r\geq 0}q^{r(r-1)}\qbin{n}{r}\qbin{k-1}{k-r}
$$
It is well-known \cite{MR1019844} (see also \cite{MR2200851} for some generalizations) that iterating the $q$-Chu-Vandermonde identity yields
\begin{align}
\qbin{n+k-1}{k}
%&=\sum_{r\geq 0}q^{r(r-1)}{n\brack r}{k-1\brack k-r} \nonumber\\
%&=\sum_{\lambda_1'\geq 0}q^{\lambda_1'(\lambda_1'-1)}{n\brack \lambda_1'}\sum_{\lambda_2'\geq 0}
%q^{\lambda_2'(\lambda_2'-1)}{\lambda_1'\brack \lambda_2'}
%{k-\lambda_1'-1\brack k-\lambda_1'-\lambda_2'} \nonumber \\
%&=\sum_{\lambda_1'\geq \lambda_2'\geq \lambda_3'\geq 0}
%q^{\lambda_1'(\lambda_1'-1)+\lambda_2'(\lambda_2'-1)+\lambda_3'(\lambda_3'-1)}
%{n\brack \lambda_1'}{\lambda_1'\brack \lambda_2'}{k-\lambda_1'-\lambda_2'-1\brack k-\lambda_1'-\lambda_2'-\lambda_3'}\nonumber \\
&=\sum_{|\lambda|=k,\ell(\lambda)\leq n}q^{2n(\lambda)}{n\brack e_0,e_1,\ldots}_q.
\label{eq:lemma}
\end{align}

Using the formula \cite[Theorem 3.3]{MR1634067}
$$
\frac{1}{(z;q)_N}=\sum_{j=0}^\infty{N+j-1\brack j}_qz^j
$$
to expand $1/(z;q)_N$ and then extracting the coefficient of $x^{m-l-1}$ in the identity
$$
\frac{1}{(x;q)_{p+1}} \frac{1}{(xq^{p+1};q)_{n-1}}=\frac{1}{(x;q)_{p+n}},
$$
we obtain
$$
\sum_{t\geq 0}{p+t\brack t}_q{n+m-3-l-t\brack m-1-l-t}_qq^{(p+1)(m-1-l-t)}=
{n+p+m-2-l\brack m-1-l}_q.
$$
Shifting $t$ to $t-l$ we get
\begin{equation}
\label{eq:simple}
\sum_{t\geq 0}{p+t-l\brack p}_q{n+m-3-t\brack n-2}_qq^{(p+1)(m-1-t)}=
{n+p+m-2-l\brack m-1-l}_q.
\end{equation}

If $\lambda=1^{e_1}2^{e_2}\cdots$, letting $\mu=1^{e_1}2^{e_2}\cdots i^{e_h-1}\cdots$ be the partition obtained by deleting part $i$ from $\lambda$, then
$$
n(\lambda) -\sum_{k=1}^i (\lambda_k'-1) =\sum_{k=1}^i {\lambda_k'-1\choose 2}+\sum_{k\geq i+1} {\lambda_k'\choose 2}= n(\mu).
$$
Hence, by replacing $e_h$ with $e_h+1$, the left-hand side of \eqref{eq:qbis}  is equal to
\begin{align*}
&\sum_{i} q^{(p+1)(m-1-i)} \qbin{p+i-l}{p}\sum_{|\mu|=m-i-1 \atop \ell(\mu)\leq n-1} q^{ 2 n(\mu)} {n-1 \brack e_0,e_1\ldots}_q \\
=&\sum_{i} q^{(p+1)(m-1-i)} \qbin{p+i-l}{p} \qbin{n+m-3-i}{n-2}\tag{by \eqref{eq:lemma}},
\end{align*}
which is the right-hand side of \eqref{eq:qbis} by \eqref{eq:simple}.

\rmk
Since the $q$-Chu-Vandermonde identity can be explained bijectively using Ferrers diagram \cite[Chapter~3]{MR1634067}, we can give a {\em bijective proof} of \eqref{eq:qbis}. Here we just sketch such a proof.
Since it is known \cite[Theorem~3.1]{MR1634067} that
$$
\qbin{M+N}{N} = \sum_{\lambda} q^{\abs{\lambda}},
$$
where $\lambda$ runs over partitions in an $M \times N$ rectangle,
the right-hand side of \eqref{eq:qbis} equals the generating function $\sum_{\lambda} q^{\abs{\lambda}}$ for all partitions $\lambda$ in an $(m-1-l) \times (n-1+p)$ rectangle.
The diagram of such a partition $\lambda$ can be decomposed as in Figure~\ref{fig:lattice}.
\begin{figure}[t]
\centering
\begin{pgfpicture}{-2.00mm}{22.14mm}{152.00mm}{95.14mm}
\pgfsetxvec{\pgfpoint{1.00mm}{0mm}}
\pgfsetyvec{\pgfpoint{0mm}{1.00mm}}
\color[rgb]{0,0,0}\pgfsetlinewidth{0.30mm}\pgfsetdash{}{0mm}
\pgfsetdash{{0.30mm}{0.50mm}}{0mm}\pgfmoveto{\pgfxy(48.00,50.00)}\pgflineto{\pgfxy(102.00,50.00)}\pgfstroke
\pgfsetdash{}{0mm}\pgfmoveto{\pgfxy(50.00,30.00)}\pgflineto{\pgfxy(100.00,30.00)}\pgflineto{\pgfxy(100.00,90.00)}\pgflineto{\pgfxy(50.00,90.00)}\pgfclosepath\pgfstroke
\pgfsetdash{{0.30mm}{0.50mm}}{0mm}\pgfmoveto{\pgfxy(70.00,92.00)}\pgflineto{\pgfxy(70.00,28.00)}\pgfstroke
\pgfsetdash{}{0mm}\pgfsetlinewidth{0.15mm}\pgfline{\pgfxy(71.42,70.00)}{\pgfxy(70.00,71.42)}\pgfline{\pgfxy(74.25,70.00)}{\pgfxy(70.00,74.25)}\pgfline{\pgfxy(77.08,70.00)}{\pgfxy(70.00,77.08)}\pgfline{\pgfxy(79.91,70.00)}{\pgfxy(70.00,79.91)}\pgfline{\pgfxy(82.74,70.00)}{\pgfxy(70.00,82.74)}\pgfline{\pgfxy(85.56,70.00)}{\pgfxy(70.00,85.56)}\pgfline{\pgfxy(88.39,70.00)}{\pgfxy(70.00,88.39)}\pgfline{\pgfxy(90.00,71.22)}{\pgfxy(71.22,90.00)}\pgfline{\pgfxy(90.00,74.05)}{\pgfxy(74.05,90.00)}\pgfline{\pgfxy(90.00,76.88)}{\pgfxy(76.88,90.00)}\pgfline{\pgfxy(90.00,79.71)}{\pgfxy(79.71,90.00)}\pgfline{\pgfxy(90.00,82.53)}{\pgfxy(82.53,90.00)}\pgfline{\pgfxy(90.00,85.36)}{\pgfxy(85.36,90.00)}\pgfline{\pgfxy(90.00,88.19)}{\pgfxy(88.19,90.00)}
\pgfsetlinewidth{0.30mm}\pgfmoveto{\pgfxy(70.00,70.00)}\pgflineto{\pgfxy(90.00,70.00)}\pgflineto{\pgfxy(90.00,90.00)}\pgflineto{\pgfxy(70.00,90.00)}\pgfclosepath\pgfstroke
\pgfsetlinewidth{0.15mm}\pgfline{\pgfxy(72.11,58.00)}{\pgfxy(70.00,60.11)}\pgfline{\pgfxy(74.94,58.00)}{\pgfxy(70.00,62.94)}\pgfline{\pgfxy(77.76,58.00)}{\pgfxy(70.00,65.76)}\pgfline{\pgfxy(80.59,58.00)}{\pgfxy(70.00,68.59)}\pgfline{\pgfxy(82.00,59.42)}{\pgfxy(71.42,70.00)}\pgfline{\pgfxy(82.00,62.25)}{\pgfxy(74.25,70.00)}\pgfline{\pgfxy(82.00,65.08)}{\pgfxy(77.08,70.00)}\pgfline{\pgfxy(82.00,67.91)}{\pgfxy(79.91,70.00)}
\pgfsetlinewidth{0.30mm}\pgfmoveto{\pgfxy(70.00,58.00)}\pgflineto{\pgfxy(82.00,58.00)}\pgflineto{\pgfxy(82.00,70.00)}\pgflineto{\pgfxy(70.00,70.00)}\pgfclosepath\pgfstroke
\pgfmoveto{\pgfxy(87.00,79.00)}\pgflineto{\pgfxy(87.00,79.00)}\pgfstroke
\pgfsetlinewidth{0.15mm}\pgfline{\pgfxy(70.45,54.00)}{\pgfxy(70.00,54.45)}\pgfline{\pgfxy(73.28,54.00)}{\pgfxy(70.00,57.28)}\pgfline{\pgfxy(74.00,56.11)}{\pgfxy(72.11,58.00)}
\pgfsetlinewidth{0.30mm}\pgfmoveto{\pgfxy(70.00,54.00)}\pgflineto{\pgfxy(74.00,54.00)}\pgflineto{\pgfxy(74.00,58.00)}\pgflineto{\pgfxy(70.00,58.00)}\pgfclosepath\pgfstroke
\pgfsetlinewidth{0.15mm}\pgfline{\pgfxy(72.00,52.45)}{\pgfxy(70.45,54.00)}
\pgfsetlinewidth{0.30mm}\pgfmoveto{\pgfxy(70.00,52.00)}\pgflineto{\pgfxy(72.00,52.00)}\pgflineto{\pgfxy(72.00,54.00)}\pgflineto{\pgfxy(70.00,54.00)}\pgfclosepath\pgfstroke
\pgfsetlinewidth{0.90mm}\pgfmoveto{\pgfxy(98.00,90.00)}\pgflineto{\pgfxy(50.00,90.00)}\pgflineto{\pgfxy(50.00,32.00)}\pgflineto{\pgfxy(56.00,32.00)}\pgflineto{\pgfxy(56.00,36.00)}\pgflineto{\pgfxy(60.00,36.00)}\pgflineto{\pgfxy(60.00,42.00)}\pgflineto{\pgfxy(68.00,42.00)}\pgflineto{\pgfxy(68.00,46.00)}\pgflineto{\pgfxy(70.00,46.00)}\pgflineto{\pgfxy(70.00,50.00)}\pgflineto{\pgfxy(72.00,50.00)}\pgflineto{\pgfxy(72.00,52.00)}\pgflineto{\pgfxy(74.00,52.00)}\pgflineto{\pgfxy(74.00,54.00)}\pgflineto{\pgfxy(76.00,54.00)}\pgflineto{\pgfxy(76.00,56.00)}\pgflineto{\pgfxy(80.00,56.00)}\pgflineto{\pgfxy(80.00,58.00)}\pgflineto{\pgfxy(84.00,58.00)}\pgflineto{\pgfxy(84.00,62.00)}\pgflineto{\pgfxy(88.00,62.00)}\pgflineto{\pgfxy(88.00,68.00)}\pgflineto{\pgfxy(90.00,68.00)}\pgflineto{\pgfxy(90.00,74.00)}\pgflineto{\pgfxy(94.00,74.00)}\pgflineto{\pgfxy(94.00,82.00)}\pgflineto{\pgfxy(98.00,82.00)}\pgflineto{\pgfxy(98.00,90.00)}\pgflineto{\pgfxy(98.00,90.00)}\pgfstroke
\pgfputat{\pgfxy(60.00,25.00)}{\pgfbox[bottom,left]{\fontsize{11.38}{13.66}\selectfont \makebox[0pt]{$p$}}}
\pgfputat{\pgfxy(48.00,39.00)}{\pgfbox[bottom,left]{\fontsize{11.38}{13.66}\selectfont \makebox[0pt][r]{$i-l$}}}
\pgfputat{\pgfxy(85.00,25.00)}{\pgfbox[bottom,left]{\fontsize{11.38}{13.66}\selectfont \makebox[0pt]{$n-1$}}}
\pgfputat{\pgfxy(48.00,69.00)}{\pgfbox[bottom,left]{\fontsize{11.38}{13.66}\selectfont \makebox[0pt][r]{$m-i-1$}}}
\pgfputat{\pgfxy(60.00,69.00)}{\pgfbox[bottom,left]{\fontsize{11.38}{13.66}\selectfont \makebox[0pt]{$q^{p(m-i-1)}$}}}
\pgfputat{\pgfxy(80.00,79.00)}{\pgfbox[bottom,left]{\fontsize{11.38}{13.66}\selectfont \makebox[0pt]{$q^{\mu_1^2}$}}}
\pgfputat{\pgfxy(75.00,63.00)}{\pgfbox[bottom,left]{\fontsize{11.38}{13.66}\selectfont \makebox[0pt]{$q^{\mu_2^2}$}}}
\pgfputat{\pgfxy(110.00,79.00)}{\pgfbox[bottom,left]{\fontsize{11.38}{13.66}\selectfont ${n-1 \brack \mu_1}_q$}}
\pgfputat{\pgfxy(110.00,63.00)}{\pgfbox[bottom,left]{\fontsize{11.38}{13.66}\selectfont ${\mu_1 \brack \mu_2}_q$}}
\pgfputat{\pgfxy(110.00,54.00)}{\pgfbox[bottom,left]{\fontsize{11.38}{13.66}\selectfont $\vdots$}}
\pgfputat{\pgfxy(71.00,55.00)}{\pgfbox[bottom,left]{\fontsize{11.38}{13.66}\selectfont $\vdots$}}
\pgfputat{\pgfxy(110.00,39.00)}{\pgfbox[bottom,left]{\fontsize{11.38}{13.66}\selectfont ${p+i-l \brack p}_q$}}
\pgfsetlinewidth{0.30mm}\pgfmoveto{\pgfxy(92.00,80.00)}\pgflineto{\pgfxy(108.00,80.00)}\pgfstroke
\pgfmoveto{\pgfxy(108.00,80.00)}\pgflineto{\pgfxy(105.20,80.70)}\pgflineto{\pgfxy(105.20,79.30)}\pgflineto{\pgfxy(108.00,80.00)}\pgfclosepath\pgffill
\pgfmoveto{\pgfxy(108.00,80.00)}\pgflineto{\pgfxy(105.20,80.70)}\pgflineto{\pgfxy(105.20,79.30)}\pgflineto{\pgfxy(108.00,80.00)}\pgfclosepath\pgfstroke
\pgfmoveto{\pgfxy(86.00,64.00)}\pgflineto{\pgfxy(108.00,64.00)}\pgfstroke
\pgfmoveto{\pgfxy(108.00,64.00)}\pgflineto{\pgfxy(105.20,64.70)}\pgflineto{\pgfxy(105.20,63.30)}\pgflineto{\pgfxy(108.00,64.00)}\pgfclosepath\pgffill
\pgfmoveto{\pgfxy(108.00,64.00)}\pgflineto{\pgfxy(105.20,64.70)}\pgflineto{\pgfxy(105.20,63.30)}\pgflineto{\pgfxy(108.00,64.00)}\pgfclosepath\pgfstroke
\pgfmoveto{\pgfxy(58.00,40.00)}\pgflineto{\pgfxy(108.00,40.00)}\pgfstroke
\pgfmoveto{\pgfxy(108.00,40.00)}\pgflineto{\pgfxy(105.20,40.70)}\pgflineto{\pgfxy(105.20,39.30)}\pgflineto{\pgfxy(108.00,40.00)}\pgfclosepath\pgffill
\pgfmoveto{\pgfxy(108.00,40.00)}\pgflineto{\pgfxy(105.20,40.70)}\pgflineto{\pgfxy(105.20,39.30)}\pgflineto{\pgfxy(108.00,40.00)}\pgfclosepath\pgfstroke
\pgfsetlinewidth{0.15mm}\pgfline{\pgfxy(71.62,50.00)}{\pgfxy(70.00,51.62)}
\pgfsetlinewidth{0.30mm}\pgfmoveto{\pgfxy(70.00,50.00)}\pgflineto{\pgfxy(72.00,50.00)}\pgflineto{\pgfxy(72.00,52.00)}\pgflineto{\pgfxy(70.00,52.00)}\pgfclosepath\pgfstroke
\pgfsetlinewidth{0.15mm}\pgfline{\pgfxy(50.00,89.60)}{\pgfxy(50.40,90.00)}\pgfline{\pgfxy(50.00,86.77)}{\pgfxy(53.23,90.00)}\pgfline{\pgfxy(50.00,83.94)}{\pgfxy(56.06,90.00)}\pgfline{\pgfxy(50.00,81.11)}{\pgfxy(58.89,90.00)}\pgfline{\pgfxy(50.00,78.28)}{\pgfxy(61.72,90.00)}\pgfline{\pgfxy(50.00,75.46)}{\pgfxy(64.54,90.00)}\pgfline{\pgfxy(50.00,72.63)}{\pgfxy(67.37,90.00)}\pgfline{\pgfxy(50.00,69.80)}{\pgfxy(70.00,89.80)}\pgfline{\pgfxy(50.00,66.97)}{\pgfxy(70.00,86.97)}\pgfline{\pgfxy(50.00,64.14)}{\pgfxy(70.00,84.14)}\pgfline{\pgfxy(50.00,61.31)}{\pgfxy(70.00,81.31)}\pgfline{\pgfxy(50.00,58.49)}{\pgfxy(70.00,78.49)}\pgfline{\pgfxy(50.00,55.66)}{\pgfxy(70.00,75.66)}\pgfline{\pgfxy(50.00,52.83)}{\pgfxy(70.00,72.83)}\pgfline{\pgfxy(50.00,50.00)}{\pgfxy(70.00,70.00)}\pgfline{\pgfxy(52.83,50.00)}{\pgfxy(70.00,67.17)}\pgfline{\pgfxy(55.66,50.00)}{\pgfxy(70.00,64.34)}\pgfline{\pgfxy(58.49,50.00)}{\pgfxy(70.00,61.51)}\pgfline{\pgfxy(61.31,50.00)}{\pgfxy(70.00,58.69)}\pgfline{\pgfxy(64.14,50.00)}{\pgfxy(70.00,55.86)}\pgfline{\pgfxy(66.97,50.00)}{\pgfxy(70.00,53.03)}\pgfline{\pgfxy(69.80,50.00)}{\pgfxy(70.00,50.20)}
\pgfcircle[fill]{\pgfxy(92.00,80.00)}{1.00mm}
\pgfsetlinewidth{0.30mm}\pgfcircle[stroke]{\pgfxy(92.00,80.00)}{1.00mm}
\pgfcircle[fill]{\pgfxy(86.00,64.00)}{1.00mm}
\pgfcircle[stroke]{\pgfxy(86.00,64.00)}{1.00mm}
\pgfcircle[fill]{\pgfxy(58.00,40.00)}{1.00mm}
\pgfcircle[stroke]{\pgfxy(58.00,40.00)}{1.00mm}
\pgfputat{\pgfxy(49.00,27.00)}{\pgfbox[bottom,left]{\fontsize{11.38}{13.66}\selectfont \makebox[0pt][r]{$(0,0)$}}}
\pgfputat{\pgfxy(101.00,90.00)}{\pgfbox[bottom,left]{\fontsize{11.38}{13.66}\selectfont $(n-1+p,m-1-l)$}}
\pgfmoveto{\pgfxy(50.00,50.00)}\pgflineto{\pgfxy(70.00,50.00)}\pgflineto{\pgfxy(70.00,90.00)}\pgflineto{\pgfxy(50.00,90.00)}\pgfclosepath\pgfstroke
\end{pgfpicture}%
\caption{Decompostion of a partition $\lambda$ in an $(m-1-l) \times (n-1+p)$ rectangle}
\label{fig:lattice}
\end{figure}
Given such a partition $\lambda$, defining $i=m-\lambda'_{p+1}-1$, we take the rectangle of size $(m-i-1) \times p$ from the point $(0,m-1-l)$ in the diagram.
And then associate a partition $\mu=(\mu_1,\mu_2,\ldots)$ of $m-i-1$ by taking the lengths $\mu_j$ of {\em successive Durfee squares}, which are started from the point $(p,m-1-l)$ and taken downwards.
Given $i$ and $\mu$, the generating function $\sum_{\lambda} q^{\abs{\lambda}}$ for all corresponding $\lambda$ is
$$ q^{p(m-i-1) +\mu_1^2 +\mu_2^2 +\mu_3^2 +\cdots}\qbin{p+i-l}{p} \qbin{n-1}{\mu_1} \qbin{\mu_1}{\mu_2} \qbin{\mu_2}{\mu_3} \cdots$$
as indicated by Figure~\ref{fig:lattice} and it follows that
\begin{align*}
&\qbin{n+m-2+p-l}{n-1+p} \\
&= \sum_{i} \sum_{n-1 \geq \mu_1\geq \mu_2\geq\cdots \atop \mu_1 +\mu_2 +\cdots = m-i-1} q^{p(m-i-1) +\mu_1^2 +\mu_2^2 +\mu_3^2 +\cdots}\qbin{p+i-l}{p} \qbin{n-1}{\mu_1} \qbin{\mu_1}{\mu_2} \qbin{\mu_2}{\mu_3} \cdots.
\end{align*}
Replacing $\mu_j$ to $\lambda_j'-1$ for $j\leq i$ (and $\mu_j$ to $\lambda_j'$ for $j>i$), the formula above is equivalent to \eqref{eq:qbis}.
Hence, the successive Durfee square decomposition of a Ferrers diagram gives a bijective proof of \eqref{eq:gen}, \eqref{eq:lemma}, and \eqref{eq:simple}.

\section{An open problem}
By \cite[Eq. (8)]{MR1928786} (see also \cite[Theorem 4]{MR2019276}), we obtain  the generating function for trees with respect to
local indegree type:
\begin{equation}\label{gf}
P_n(x_1,\ldots,x_n)=\sum_{T\in \Tn} \prod_{i=1}^{n} x_i^ {\indeg_T(i)}
= x_n \prod_{i=2}^{n-1} (i x_i + x_{i+1} +\cdots + x_n),
\end{equation}
where $\indeg_T(i)$ is the indegree of vertex$i$ in $T$ with the local orientation.
We say that a monomial $\mathbf{x}^{\boldsymbol{\alpha}} = x_1^{\alpha_1}x_2^{\alpha_2}\ldots x_n^{\alpha_n}$ is of type $\lambda = 1^{e_1} 2^{e_2} \ldots$ if the sequence $\boldsymbol{\alpha}=(\alpha_1,\ldots, \alpha_n)$ has $e_h$ $i$'s for $0< i \le n$.
For any partition $\lambda = 1^{e_1}2^{e_2} \cdots$ of $n-1$ and $e_0 = n-\ell(\lambda)$, from \eqref{first} and \eqref{gf} we derive
\begin{equation}\label{conj}
\sum_{\type(\mathbf{x}^{\boldsymbol{\alpha}})=\lambda}
[\mathbf{x}^{\boldsymbol{\alpha}}] P_n(x_1,\ldots,x_n) =
\dfrac{(n-1)!^2}{e_0!(0!)^{e_0} e_1! (1!)^{e_1} e_2! (2!)^{e_2} \ldots}, %a_{\lambda}
\end{equation}
where $[\mathbf{x}^{\boldsymbol{\alpha}}] P_n(x_1,\ldots,x_n)$ denotes the coefficient of $\mathbf{x}^{\boldsymbol{\alpha}}$ in $P_n(x_1,\ldots,x_n)$.

For example, if $n=4$, the generating function reads as follows:
$$P_4(x_1,x_2,x_3,x_4)
%&=& x_4 (2x_{2} + x_{3} +x_4) (3x_{3} +x_4)\\
 = 6 x_2 x_3 x_4 + 2 x_2 x_4^2 + 3 x_3^2 x_4 + 4 x_3 x_4^2 + x_4^3.
$$
Clearly, the monomials of type $\lambda=1^1 2^1$ are $x_2 x_4^2$, $x_3^2 x_4$ and $x_3 x_4^2$ and the sum of their coefficients is
$
%\sum_{\type(\mathbf{x}^{\boldsymbol{\alpha}})=\lambda} [\mathbf{x}^{\boldsymbol{\alpha}}] P_4(x_1,x_2,x_3,x_4) =
2+ 3+ 4 = 9,
$
which coincides with the formula  \eqref{first}, i.e., ${3!^2}/{2!^2} =9$.

\subtitle{Open problem}
Find a {\em direct proof} of the algebraic identity \eqref{conj}.
%It would be interesting to have a {\em direct proof} of \eqref{conj}.
%erive Cotterill's formula from the above generating function.

\subtitle{Acknowledgement}
We are grateful to  the two referees for valuable suggestions on a previous version and Victor Reiner for informing us the two references \cite{MR1928786, MR2019276}. %the generating function \eqref{gf}.
This work was partially supported by the Korea Research Foundation Grant funded by the Korean Government(MOEHRD). KRF-2007-357-C00001.

%\bibliographystyle{amsabbrv}
%\bibliography{SZ08a}

\providecommand{\bysame}{\leavevmode\hbox to3em{\hrulefill}\thinspace}
\providecommand{\href}[2]{#2}

\end{document}